\documentclass[12pt]{amsart}
 \usepackage{color}
\usepackage{amssymb,amsmath,amsthm,amsfonts}
\usepackage[all]{xy}
\calclayout
\newtheorem*{thma}{Theorem A}
\newtheorem*{thmb}{Theorem B}
\newtheorem*{thmc}{Theorem C}
\newcommand{\namedRef}[1]{%
\csname env:#1\endcsname\ \ref{#1}%
}

\newtoks\displayedcounter\displayedcounter={dummy}

\def\envLRT#1#2{\expandafter\gdef\csname env:#1\endcsname{#2}}

\makeatletter
\newcommand{\namedLabel}[1]{\expandafter\xdef\csname env:#1\endcsname{\envName}%
\immediate\write\@auxout{\string\envLRT{#1}{\envName}}%
\label{#1}}
\makeatother
\newcommand{\newNumber}[2][]{\refstepcounter{\the\displayedcounter}%
\ifx#1\empty\label{#2}\else\gdef\envName{#1}\namedLabel{#2}\fi}
\def\newType#1#2{\newtheorem{#1}[\the\displayedcounter]{#2\gdef\envName{#2}}}
\newType{lemma}{Lemma}
\newType{proposition}{Proposition}
\newType{corollary}{Corollary}
\newType{theorem}{Theorem}
\theoremstyle{definition}%%Change Theorem style
\newType{definition}{Definition}
\newType{example}{Example}
\newType{remark}{Remark}
\newcommand
{\eqncount}{\setcounter{equation}{\value{dummy}}%
\addtocounter{dummy}{1}}

\newcommand{\bbR}{\mathbb R}
\newcommand{\cA}{\mathcal A}
\newcommand{\cB}{\mathcal B}
\newcommand{\cC}{\mathcal C}
\newcommand{\cD}{\mathcal D}
\newcommand{\cE}{\mathcal E}
\newcommand{\cF}{\mathcal F}
\newcommand{\cG}{\mathcal G}
\newcommand{\cH}{\mathcal H}
\newcommand{\cI}{\mathcal I}
\newcommand{\cK}{\mathcal K}
\newcommand{\cL}{\mathcal L}
\newcommand{\cM}{\mathcal M}
\newcommand{\cN}{\mathcal N}

\newcommand{\bZ}{\mathbf Z}
\newcommand{\bA}{\mathbf A}

\newcommand{\bC}{\mathbf C}
\newcommand{\bQ}{\mathbf Q}
\newcommand{\DG}{\cD(G)}
\newcommand{\ab}{{\mathcal Ab}}
\newcommand{\Ab}{{\mathcal Ab}}
\newcommand{\fs}{\mathfrak f}
\newcommand{\fg}{\mathfrak g}

\newcommand{\disjointunion}{\bigsqcup}
\newcommand{\vv}{\, | \,}
\DeclareMathOperator{\Hom}{Hom}
\newcommand{\bd}{\partial}
\DeclareMathOperator{\Image}{Im}
\DeclareMathOperator{\invlim}{\lower 6pt\hbox{$\stackrel{\displaystyle{\lim}}{\leftarrow}$}}
\DeclareMathOperator{\pr}{pr}
\newcommand{\Zhat}{\widehat \bZ}
\newcommand{\RGMorita}{RG\text{-Morita}}
\newcommand{\RGammaMorita}{R\Gamma\text{-Morita}}
\newcommand{\RMorita}{R\text{-Morita}}
\newcommand{\RbarMorita}{(R,-)\text{-Morita}}
\newcommand{\RbarWitt}{(R,-)\text{-Witt}}
\newcommand{\RGomegaMorita}{(RG,\omega)\text{-Morita}}
\newcommand{\RGomegaWitt}{(RG,\omega)\text{-Witt}}
\DeclareMathOperator{\wh}{Wh}

\DeclareMathOperator{\Or}{\mathbf{Or}}
\DeclareMathOperator{\Ind}{Ind}
\DeclareMathOperator{\Res}{Res}
\newcommand\prd{\hbox{.} }
\newcommand{\qm}{?}

 \newcommand{\bq}[1]{\iota_{#1}}
\newif\ifmathCheck\mathChecktrue

\def\cDdot{\cD_{\ast}}
\def\bAdot{\bA_\bullet}
\def\congtran{\tau}
\def\congtrandot{\tau^{\bullet}}

\def\upj{j^{\bullet}}\def\downj{j_{\bullet}}

\newcommand{\DtoAdot}{j}

\def\DtoDdot{\mu^\prime}
\def\DdotD{\mu}

\def\Burnside{{\mathbf A}}

\newdimen\bisetdim\bisetdim=4pt
\def\biset#1#2#3{{\lower \bisetdim\hbox{$\scriptstyle#1$}}%
{{#2}}{\lower \bisetdim\hbox{$\scriptstyle#3$}}}

\begin{document}

\title[Burnside Quotient Green Ring] 
{Dress induction and the Burnside quotient \\ Green ring}
\author{I.~Hambleton}
\thanks{Research partially supported by
NSERC Discovery Grant A4000 and the NSF}
\address{Department of Mathematics \& Statistics
 \newline\indent
McMaster University
 \newline\indent
Hamilton, ON  L8S 4K1, Canada}
\email{ian{@}math.mcmaster.ca}
\author{L.~R.~Taylor}
\address{Department of Mathematics
 \newline\indent
University of Notre Dame
 \newline\indent
Notre Dame, IN 46556, USA}
\email{taylor.2@nd.edu}
\author{E.~B.~Williams}
\address{Department of Mathematics
 \newline\indent
University of Notre Dame
 \newline\indent
Notre Dame, IN 46556, USA}
\email{williams.4@nd.edu}

\date{April 27, 2009}
\begin{abstract}\noindent
We define and study the Burnside quotient Green ring of a Mackey functor, introduced in our MSRI preprint \cite{htw4}. Some refinements of Dress induction theory are presented, together with applications to computation results for $K$-theory and $L$-theory of finite and infinite groups.
\end{abstract}

\maketitle
\section{Introduction}

  Induction theory began with Artin and Brauer's work in representation theory, 
was continued by Swan \cite{swan2} and Lam \cite{lam1} for $K${-}theory, and was put in its most
abstract and elegant setting by Green \cite{green1} and Dress \cite{dress1}, \cite{dress2}.  The theory sets up a convenient framework for computing the 
 value of a Mackey functor on some finite group $G$, given suitable generation results for a Green ring which acts on the Mackey functor (see   \cite{tomDieck2}, \cite{lindner1},  \cite{thevenaz1}, \cite{thevenaz-webb2} for some of the subsequent developments in this subject).
 
 The main examples in this theory are  (i) the Swan Green ring $SW(G, \bZ)$ (see \cite{swan2}), which leads to the Brauer-Berman-Witt induction theorem for representations of finite groups, and computation results for Quillen $K$-theory $K_n(RG)$, and (ii) the Dress Green ring $GU(G, \bZ)$
(see \cite{dress2}), which leads to computation results for the oriented surgery obstruction groups
 $L_n(\bZ G)$ of C.~T.~C.~Wall \cite{wall-VI}. 
 
 In Section \ref{one} we define the Burnside quotient Green ring $\cA_\cM$ for a Mackey functor $\cM\colon \cD(G) \to \Ab$, where $\cD(G)$ denotes the  category of finite $G$-sets, and $\Ab$ the category of abelian groups. This Green ring $\cA_\cM$ is the smallest quotient of the Burnside ring which is a Green ring, and still acts on the Mackey functor. As defined, it has many convenient naturality properties,
and generation results for $\cA_\cM$ will lead as usual to computation results for $\cM$. We define the concept of  a \emph{Dress generating set} $X$ for a Green ring in Definition \ref{dress_generating}. The main result (see Theorem \ref{main_theorem}) is:
\begin{thma}
A finite $G$-set $X$ is a Dress generating set for a Green ring
$\cG$ if and only if it is a Dress generating set for the Burnside
quotient Green ring $\cA_{\cG}$. 
\end{thma}
The naturality of the Burnside quotient Green rings can now be used to obtain computability results for sub- or quotient Mackey functors (see Theorem \ref{subquotient_properties}). We also point out a useful refinement of Dress induction in Theorem \ref{dress-refinement}.
We use the Burnside quotient Green ring in Section \ref{three} to study additive functors out of the categories $\RGMorita$ defined in \cite{htw3},   where $R$ is a commutative ring with unit. The main examples of such functors include $K$-theory, 
Hochschild homology and cyclic homology (see \cite[1.A.12]{htw3}). We define a bifunctor $d\colon \cD(G) \to \RGMorita$ in (\ref{DtoBurnside}) and prove  the following computability result (see Theorem \ref{rgmorita}):
\begin{thmb} 
 Any additive functor $F\colon \RGMorita\to \Ab$ gives a Mackey functor on $\cD(G)$ by composition with $d\colon \cD(G) \to \RGMorita$. Any such Mackey functor is hyperelementary computable.
 \end{thmb}
 This is a refinement of a result of Oliver \cite[11.2]{oliver3}, and Theorem A provides the key new ingredient in the proof. The extra generality is useful for studying functors such as the Dade group and the units in the Burnside ring (see Bouc \cite{bouc3}, \cite{bouc5}).
 
  The Burnside quotient Green ring has been applied to study the permutation representations of finite groups in \cite{htaylor1}, free actions of finite groups on products of spheres in \cite{h4}, and to the computation of Bass Nil-groups in 
 \cite{hlueck1}. This theory was  surveyed and used in \cite{htaylor2}. 
 Our results also apply to the computation of $K$ and $L$-theory for infinite groups, based on an idea of Farrell and Hsiang \cite{farrell-hsiang3}.
 
 \smallskip
    We introduce  Mackey \emph{pre}-functors and \emph{pseudo}-Mackey functors in Section \ref{four}. A Mackey pre-functor is a just pre-bifunctor $\cD(G) \to \Ab$, and a pseudo-Mackey functor is a Mackey pre-functor which admits a finite filtration by Mackey functors. Such structures have been observed in a number of different contexts: the main examples include the higher Whitehead groups $\wh_n(\bZ G)$, and the structure set of a compact manifold in surgery theory (see \cite[Chap.~9]{wall-book}).

     It turns out that the general scheme of Dress induction theory can be extended to pseudo-Mackey functors as well.
 In Section \ref{five}, we combine this idea with the Burnside quotient Green ring to study additive functors out of the category $\RGomegaMorita$   (see \cite[1C]{htw3}). We have the corresponding computability result (see  Theorem \ref{thmc}):
 \begin{thmc}   
 Let $F\colon \RGomegaMorita \to \Ab$ be an additive functor. 
 Then the composite $\cM=F\circ d\colon \cD(G) \to \Ab$ is a Mackey pre-functor. Moreover:
 \begin{enumerate}
 \item The $2$-adic completion of $\cM$ is $2$-hyperelementary computable.
 \item If $\cM$ is a Mackey functor, then $\cM$ is hyperelementary computable.
 \end{enumerate}
 \end{thmc}
 As an application, we conclude from part (i) that the surgery obstruction groups $L_n(\bZ G, \omega)$, with arbitrary orientation character $\omega\colon G \to \{\pm 1\}$, are $2$-hyperelementary computable after $2$-adic completion  
 (see Example \ref{nonorientedL} for the meaning of computability in this setting). This computability result was certainly expected to be true, but the argument presented here seems to be the first actual proof in the non-oriented case.  In the oriented case, where $\omega$ is trivial, part (ii) applies to $L$-theory and the computability is just the result of Dress \cite[Theorem 1]{dress2}. For non-trivial $\omega$, the surgery obstruction group $L_n(\bZ G, \omega)$ is a Mackey functor if and only if it has exponent two (see \cite{taylor1}, and see \cite[5.2.5]{wall-VI} for an example where the $L$-groups do not have exponent two). In Lemma \ref{pretomackey} we give a general necessary and sufficient condition on $F$ for part (ii) to apply to $\cM$.

 \medskip\noindent
{\bf Acknowledgement}. We   were greatly indebted to L.~Gaunce Lewis, Jr.~for his meticulous reading of our 1990 MSRI preprint, and for his extensive helpful comments. He pointed out  to us that the definition of the Burnside quotient Green ring appears in his unpublished notes, \emph{The theory of Green functors} (1980). This paper is dedicated to his memory.

\section{Dress Induction}\label{one}
We will first recall some definitions Dress used in his formulation of
induction theory \cite[p.~301]{dress2}.
\subsection{Mackey Functors} Let $G$ be a finite group, and let 
$\DG$ denote the category whose objects are finite, left $G$-sets and whose morphisms are $G$-maps. A \emph{Mackey functor}
is a bifunctor $\cM = (\cM_\ast, \cM^\ast) \colon \cD(G) \to \ab$, where $\ab$ denotes the category of abelian groups and groups homomorphisms, such that $\cM_\ast(S) = \cM^\ast(S)$ for each object $S \in \DG$, and the following two  properties hold:

\medskip\noindent
\textbf{(M1)} For any pullback diagram of finite $G$-sets
$${\vcenter
{\xymatrix
{S\ar[r]^\Psi\ar[d]_\Phi&S_1\ar[d]^\varphi\\S_2\ar[r]^\psi&T}}}
$$
the induced maps give an commutative diagram
$${\vcenter
{\xymatrix
{\cM(S)\ar[r]^{\Psi_\cM}&\cM(S_1)\\
\cM(S_2)\ar[r]^{\psi_\cM}\ar[u]^{\Phi^\cM}&\cM(T)\ar[u]_{\varphi^\cM}}}}
$$
Here we denote the covariant maps by $\psi_\cM$ and the
contravariant maps by $\varphi^\cM$.

\medskip\noindent
\textbf{(M2)} The  embeddings of $S_1$ and $S_2$ into the disjoint
union $S_1\disjointunion S_2$ define an isomorphism
$\cM^\ast(S_1\disjointunion S_2) \to \cM^\ast(S_1)\oplus \cM^\ast(S_2)$. Let $\cM (\emptyset) = 0$.

\medskip
The property (M1) is the usual double coset formula, and (M2) gives additivity.
We remark that for any bifunctor satisfying (M1), the composition
$\cM_\ast(S_1)\oplus \cM_\ast(S_2) \to \cM_\ast(S_1\disjointunion S_2) = 
\cM^\ast(S_1\disjointunion S_2) \to \cM^\ast(S_1)\oplus \cM^\ast(M_2) = 
\cM_\ast(S_1)\oplus \cM_\ast(S_2)$ 
is just the identity matrix.
It follows that  any sub-bifunctor of a Mackey functor is Mackey. 

\begin{definition}  
If $\cM$ and $\cN$ are Mackey functors, then a \emph{homomorphism} $\cM \to \cN$ of Mackey functors is a natural transformation of bifunctors $\Theta\colon \cM \to \cN$ such that for each object
$S \in \DG$ the function $\Theta_S \colon \cM(S) \to \cN(S)$ is a homomorphism of abelian groups. It is easy to check that the kernel, 
$\ker \Theta$, the  image, $\Image \Theta$, and the cokernel of $ \Theta$ are all sub- or quotient Mackey functors of $\cM$ or $\cN$.
\end{definition}

\subsection{Pairings and Green functors} If $\cM$, $\cN$, and $\cL$ are Mackey functors, then a \emph{pairing} is a family of bilinear maps
$$\cM(S) \times \cN(S) \to \cL(S)$$
indexed by the objects of $\DG$, such that for any $G$-map
$\varphi\colon S \to T$ the following formulas hold:
\begin{eqnarray*}
\varphi^{\cL}(x \cdot y) &=&\varphi^\cM(x) \cdot \varphi^\cN(y)\quad \,\text{for\ } x\in \cM(T), y \in \cN(T)\\
x \cdot \varphi_\cN( y) &=&\varphi_\cL(\varphi^\cM(x )\cdot y) \quad \text{\ for\ } x\in \cM(T), y \in \cN(S)\\
\varphi_\cM(x) \cdot y&=&\varphi_\cL(x \cdot \varphi^N( y)) \quad\, \text{\ for\ } x\in \cM(S), y \in \cN(T)
\end{eqnarray*}

A \emph{Green ring}  is a Mackey functor $\cG$ together with a pairing $\cG \times \cG \to \cG$, and a collection of elements
$\{1_S \in \cG(S)\}$  such that the pairing defines an associative ring structure on each $\cG(S)$ with unit $1_S$, and $\varphi^\cG(1_T)=1_S$ for every $G$-map $\varphi\colon S \to T$. 

A \emph{homomorphism}  of Green rings $\Theta\colon \cG \to \cK$ is a homomorphism of Mackey functors
such that for each object
$S \in \DG$ the function $\Theta_S \colon \cG(S) \to \cK(S)$ is a unital ring homomorphism.   If $\Theta_S$ is injective for each object
$S \in \DG$, we say that $\cG$ is a \emph{sub}-Green ring of $\cK$. If $\Theta_S$ is surjective for each object
$S \in \DG$, then we say that $\cK$ is a \emph{quotient} Green ring of $\cG$.
 Similarly, we define sub-quotient Green rings.

If $\cM$ is a Mackey functor, then $\cM$ is a 
\emph{Green module} over a Green ring $\cG$ if there exists a pairing $\cG \times \cM \to \cM$ such that $\cM(S)$ becomes a left
$\cG(S)$-module from the pairing, and $1_S\cdot x = x$ for
all $x\in \cM(S)$.
\begin{example}  If  $ \cG \to \cK$ is a homomorphism of Green rings,  then $\cK$ is a Green module over $\cG$ under the pairing $\cG \times \cK \to \cK$ induced by the homomorphism. 
\end{example}

\subsection{The Burnside ring} 
For any left $G$-set $S$, we let $\cD_S(G)$ denote the category with objects $(X,f)$, where $X$ is a left $G$-set and $f\colon X \to S$ is a $G$-map. The morphisms $F\colon (X_1,f_1) \to (X_2,f_2)$ are $G$-maps $F\colon X_1\to X_2$ such that $f_2\circ F = f_1$. We define a
bifunctor $$\cA\colon \DG \to \ab$$ by setting $\cA(S) = K_0(\cD_S(G))$. 
If $\varphi\colonΠS\to T$ is a $G$-map, then $\varphi_\cA\colon \cA(S) \to \cA(T)$ is the  map induced on $K_0$ by the composition
$(X,f) \mapsto (X, \varphi\circ f)$. The contravariant map
$\varphi^\cA\colon \cA(T) \to \cA(S)$ is induced by the pullback
construction applied to $S \xrightarrow{\varphi} T \xleftarrow{f} Y$,
where $(Y,f)$ is an object in $\cD_T(G)$. The conditions (M1) and (M2)
are easy to check, and $\cA$ is a  Mackey functor. There is also a
pairing $\cA \times \cA \to \cA$ defined by pullback: let $(X_1,f_1)$
and $(X_2,f_2)$ represent elements of $\cA(S)$, and form the pullback
$X_1\xrightarrow{f_1} S \xleftarrow{f_2} X_2$ considered as $G$-set over $S$. This object in $\cD_S(G)$ represents the product, and
each $\cA(S)$ becomes an associative ring with unit element represented by $S\xrightarrow{id} S$. The resulting Green ring is called the \emph{Burnside ring}. Dress also remarks that the Burnside ring is the ``universal" Green ring, since it acts on any Mackey functor $\cM$. The required pairing $\cA \times \cM \to \cM$ is defined by pairing an element of $\cA(S)$ represented by a $G$-set $(X,f)$ over
$S$, and an element $x\in \cM(S)$, to get $f_\cM(f^{\cM}(x)) \in \cM(S)$.
It is not hard to check that $\cM(S)$ is a unital $\cA(S)$-module under this bilinear pairing, so $\cM$ is a Green module over $\cA$. 

We remark that a homomorphism  $\cM \to \cN$ of Mackey functors is compatible with the $\cA$-module action, so gives a map of
 $\cA$-Green modules.

If $\cG$ is a Green ring, the same checks show that $\cG$ is an $\cA$-algebra, implying in particular that $a\cdot(x\cdot y) = (a\cdot x)\cdot y$ for all $a\in \cA(S)$ and all $x,y\in \cG(S)$. It follows that the map $\iota \colon \cA \to \cG$
defined by $a\mapsto a \cdot 1_S$, for all $a\in \cA(S)$, is a (unital) ring homomorphism. Indeed
$$(a\cdot 1_S)\cdot (b\cdot 1_S) = a\cdot(1_S \cdot (b\cdot 1_S)) 
=a\cdot(b\cdot 1_S) = (a\cdot b) \cdot 1_S$$
for all $a,b\in \cA(S)$, since $\cG(S)$ is a $\cA(S)$-algebra. It is easy to check from the pairing formulas that $\iota \colon \cA \to \cG$ is also a  homomorphism  of Green rings.

\subsection{Ideals and quotient Green rings}  There is a natural notion of a (left) \emph{Green ideal} in a Green ring $\cG$, namely a sub-bifunctor $I \subset \cG$ such that $I(S) \subset \cG(S)$ is a left ideal in the ring $\cG(S)$.
 Similarly, we have right ideals and two-sided ideals.
If $I \subset \cG$ is a two-sided Green ideal, then the quotient functor $\cG/I$, defined by $S\mapsto \cG(S)/I(S)$, is a Green ring under the quotient pairing inherited from $\cG$.

If $\cG \times \cM\to \cM$ is a Green module structure on a Mackey
functor $\cM$, then we define the \emph{Green ideal} $I_\cM \subset
\cG$ as the sub-bifunctor of $\cG$ with 
$$I_\cM(S) = \{ a\in \cG(S) \vv \varphi_\cG(a)\cdot y = 0, \psi^\cG(a)\cdot z =0\}$$
for all $\varphi\colon S\to T$, $\psi\colon U \to S$, and all
$y \in \cM(T)$, $z\in \cM(U)$. Note that elements of $I_\cM(S)$ satisfy additional conditions (both ``up" and ``down") beyond just acting trivially on $\cM(S)$.

The pairing formulas show directly
that $I_\cM(S)$ is a two-sided ideal in the ring $\cG(S)$, for every finite $G$-set $S$. 
We will check that $I_\cM$ is a sub-bifunctor
of $\cG$ by looking at the operations induced by $G$-maps
$\mu\colon V \to S$ and $\lambda\colon S \to W$ on an arbitrary element $a\in I_\cM(S)$. 

First  we consider $\lambda_\cG(a)\in \cG(W)$. Let $\varphi\colon
W\to T$ and $\psi\colon U \to W$ be any $G$-maps. We have
$$\varphi_\cG(\lambda_\cG(a))\cdot y= (\varphi\circ \lambda)_\cG(a)\cdot y = 0$$
by definition of $I_\cM(S)$. Let 
$${\vcenter
{\xymatrix
{X\ar[r]^{\tilde\lambda}\ar[d]_{\tilde\psi}&U\ar[d]^\psi\\S\ar[r]^\lambda&W}}}
$$
be the pullback square, and from (M1) we get
$$\psi^\cG(\lambda_\cG(a))\cdot z = \tilde\lambda_\cG(\tilde\psi^\cG(a))\cdot z = \tilde\lambda_\cG(\tilde\psi^\cG(a)\cdot \tilde\lambda^\cM(z)) = 0$$
so $\lambda_\cG(a) \in I_\cM(W)$.

Similarly, we must check that $\mu^\cG(a) \in I_\cM(V)$. Let 
$\varphi\colon V \to T$ and $\psi\colon U \to V$ be $G$-maps, and
note that 
$$\varphi_\cG(\mu^\cG(a))\cdot y = \varphi_\cM(\mu^\cG(a)\cdot \varphi^\cM(y))= 0$$
and $\psi^\cG(\mu^\cG(a))\cdot z = (\mu\circ \psi)^\cG(a) \cdot z = 0$.

We have now checked that $I_\cM\subset \cG$ is a sub-bifunctor, and
therefore $I_\cM$ is a Mackey functor and  a two-sided Green ideal in $\cG$. We define the \emph{quotient Green ring} $\cG_\cM = \cG/I_\cM$ to be the bifunctor whose value on objects is given by the quotient rings $\cG_\cM(S) = \cG(S)/I_\cM(S)$. It is straightforward to check that $\cG_\cM$ is a Green ring, since the formulas above show that the pairing $\cG \times \cG \to \cG$ restricts to pairings $I_\cM \times \cG \to I_\cM$ and
$\cG\times I_\cM \to I_\cM$ of Mackey functors. By construction,
$\cM$ is also a Green module over $\cG_\cM$.

\begin{definition}\label{definition of BQGR} 
Let $\cM$ be a Mackey functor. The \emph{Burnside quotient Green ring} of $\cM$ is the Green ring $\cA_\cM:=\cA/I_\cM$.
Let $\bq{\cM} \colon \cA \to \cA_{\cM}$ denote the epimorphism of Green rings given by the natural quotient map. 
\end{definition}

\begin{remark}
For $\cG$ a Green ring, the map $\iota\colon \cA\to \cG$ defined above by $a\mapsto a\cdot 1_S$ factors through $\bq{\cG}\colon \cA \to \cA_\cG$, and we
obtain a \emph{canonical} induced
homomorphism of Green rings  $\cA_{\cG} \to \cG$.   The next result shows that in fact $\cA_{\cG}=\Image\iota$, which gives a quick alternate definition of $\cA_\cG$  (for this observation, compare \cite[p.~207]{dress1}, \cite[p.~253]{oliver3}, \cite[p.~711]{htaylor2}, and \cite[p.~236]{bak2}).
\end{remark}

\begin{lemma} Let $\cG$ be a Green ring. Then the canonical homomorphism of Green rings, 
$\cA_\cG \to \cG$ is injective. 
\end{lemma}
\begin{proof} For each $G$-set $S$, the natural transformation of bifunctors in the statement maps $\cA_\cG(S)\to \cG(S)$  by the ring homomorphism $a\mapsto a\cdot 1_S$, where $a\in \cA(S)$ and $1_S\in \cG(S)$ is the unit.
If $a\cdot 1_S =0$, and $\varphi\colon S\to T$, $\psi\colon U \to S$
are $G$-maps, then it follows as above that $\varphi_\cA(a)\cdot 1_T =0$ and $\psi^\cA(a)\cdot 1_U = 0$, Therefore 
$\{a\in \cA(S)\vv a\cdot 1_S = 0\}\subset I_\cG(S)$, and the ring homomorphism
$\cA_\cG(S)\to \cG(S)$ is injective. 
\end{proof}

We will explore Definition \ref{definition of BQGR} 
 by considering the Burnside quotient Green rings for filtrations of Mackey functors. 

\begin{definition}  If $\cM$ and $\cN$ are Mackey functors, we say that
$\cM$ is a \emph{sub-functor} of $\cN$ (respectively $\cN$ is a \emph{quotient functor} of $\cM$) if there is a natural transformation $\Theta\colon \cM \to \cN$ such that for each object
$S \in \DG$ the function $\Theta_S \colon \cM(S) \to \cN(S)$ is an injective 
(respectively, surjective) homomorphism of abelian groups. We say that $\cM$ is a \emph{sub-quotient}
of $\cN$  if there is a finite sequence of Mackey functors $\cM=\cL_0,
\cL_1, \dots, \cL_r=\cN$ such that each $\cL_i$ is either a sub-functor or a quotient functor of $\cL_{i+1}$, for $i = 0, \dots, r-1$. Note that
the relation ``$\cM$ is a sub-quotient of $\cN$" is a transitive relation.
\end{definition}

\begin{example}    If $\Theta\colon \cG \to \cK$ is a  homomorphism of Green rings, then we can regard $\cK$ as a Green module over $\cG$. Furthermore, 
$\ker\Theta = I_{\cK}\subset \cG$, 
and there is an induced homomorphism $\cG_{\cK} \to \cK$ of Green rings.  If $\cK$ is a quotient Green ring of $\cG$, then  $\cK =  \cG_{\cK} =\cG/I_{\cK}$.
\end{example}

\begin{lemma} Let $\cG$ be a Green ring and $\cM$ a Green module over $\cG$. Then the Burnside quotient Green ring $\cA_M$ is a quotient of $\cA_\cG$, and isomorphic to a sub-quotient of $\cG$.
\end{lemma}
\begin{proof}
Since $\cA_\cG$ is a sub-Green ring of $\cG$, we just need to check that  $\cA_\cM$ is a quotient Green ring of $\cA_\cG$ under the natural projection from $\cA$. This is equivalent to the statement that $I_\cG \subset I_\cM$. Let $a\in I_\cG(S)$, and consider $G$-maps $\varphi\colon S\to T$ and $\psi\colon U \to S$.
For any $y \in \cM(T)$, 
$$\varphi_\cA(a)\cdot y = \varphi_\cA(a)\cdot(1_T\cdot y) = 
(\varphi_\cA(a)\cdot 1_T)\cdot y = 0$$
since $1_T \in \cG(T)$.  Similarly, for any $z\in \cM(U)$, 
$$\psi^\cA(a)\cdot z = \psi^\cA(a)\cdot 1_U\cdot z = 0$$
and we see that $a\in I_\cM(S)$.
\end{proof}
\begin{lemma}  
 Let $\cM$ and $\cN $ be Mackey functors, with $\cM$  a sub-quotient
 of $\cN$. 
 Then there is a surjective homomorphism of Green rings $\fs\colon \cA_\cN \to \cA_\cM$ such that
 $\fs\circ \bq{\cN}  = \bq{\cM}$.
\end{lemma}
\begin{proof} We will establish this result for sub-functors and quotient functors, and note that the general sub-quotient case follows by an inductive argument on the length of the chain joining $\cM$ to 
$\cN$. 

Suppose first that   $\Theta \colon \cM \to \cN$ is a natural transformation, with $\Theta_S\colon \cM(S) \to \cN(S)$ injective for all finite $G$-sets $S$. Let $a\in I_\cN(S)$
and let $\varphi\colon S\to T$, $\psi\colon U \to S$ be $G$-maps.
Then for any $y \in \cM(T)$, $\Theta_T(\varphi_\cA(a)\cdot y)=
\varphi_\cA(a)\cdot(\Theta_T(y))=0$ since $\Theta$ is a $\cA$-Green module map. Similarly, for any $z\in \cM(U)$,
$\Theta_U(\psi^\cA(a)\cdot z)= \psi^\cA(a)\cdot(\Theta_U(z))=0$.
Since $\Theta_T$ and $\Theta_U$ are injective, it follows that
$a\in I_\cM(S)$, and $I_\cN \subset I_\cM$ so that $\cA_\cN$ maps onto $\cA_\cM$.

Next suppose that 
$\Theta \colon \cN \to \cM$ is a natural transformation, with $\Theta_S\colon \cN(S) \to \cM(S)$ surjective for all finite $G$-sets $S$. If $a\in I_\cN(S)$, we check that $\varphi_\cA(a)\cdot y =0 $ and $\psi^\cA(a)\cdot z=0$, for all 
$y \in\cM(T)$ and all $z\in \cM(U)$, by using the surjectivity of $\Theta_T$ and $\Theta_U$, and the compatibility of
$\Theta$ with the $\cA$-module structures on $\cM$ and $\cN$.
Therefore $I_\cN \subset I_\cM$.
\end{proof}

In general, if $\cM$ is a sub-Mackey functor of $\cN$ it is not true that $I_\cM \subset I_\cN$,  so there is no natural map in the other direction from
$\cA_\cM$ onto $\cA_\cN$, but here is one more situation that works.  

We say that $\cM$ is
a \emph{full lattice} in $\cN$ if there is a natural transformation
$\Theta\colon \cM \to \cN$ such that the induced maps
$\Theta_S^\ast\colon \Hom(\cN(S) ,\cN(S)) \to \Hom(\cM(S), \cN(S))$
are injective for all finite $G$-sets $S$. Note that $\cM$ need not be a sub-Mackey functor of $\cN$ for this condition to hold. 

\begin{lemma}   Let  $\cM$ and $\cN$ be Mackey functors, with $\cM$ be a full lattice in $\cN$. Then 
there exists a surjective homomorphism of Green rings $\fg\colon \cA_\cM \to \cA_\cN$ such that
$\fg\circ \bq{M} = \bq{N}$.
If $\cM$ is also a sub-functor of $\cN$, then $\fg$ is an isomorphism and the inverse to the $\fs\colon \cA_\cN \to \cA_\cM$ described
previously.
\end{lemma}
\begin{proof}  Let $\varphi\colon S \to T$ be a $G$-map. For each $a\in\cA(S)$ we can consider the action map
$y \mapsto \varphi_\cA(a)\cdot y$ as an element of $\Hom(\cN(T),\cN(T))$. However if  $a\in I_\cM(a)$, this homomorphism is zero on the image of $\Theta_T$, and therefore it vanishes identically. Similarly, we check that $\psi^\cA(a)\cdot z = 0$ for all $z\in \cN(U)$ and any $G$-map $\psi\colon U \to S$. Therefore
$I_\cM \subset I_\cN$.
\end{proof}

\subsection{Amitsur complexes}
Dress proves computation results for Mackey functors via the contractibility of certain chain complexes. Let $X$, $Y$ be finite $G$-sets, and define a semi-simplicial set $Am(X,Y)$ inductively. 
Let $Am_0(X,Y) = Y$ and $Am_r(X,Y) = X \times Am_{r-1}(X,Y)$
for $r\geq 1$. There are $G$-maps $$d_i^r\colon Am_r(X,Y) \to Am_{r-1}(X,Y)$$ for $0\leq i < r$, defined by setting $d_0^r$ as the projection
$X \times Am_{r-1}(X,Y) \to Am_{r-1}(X,Y)$, and for $i>0$ by
$d^r_i= 1_X \times d^{r-1}_{i-1}$. 

\begin{definition} Let $\cM$ be a Mackey functor. For given finite $G$-sets $X$, $Y$, 
the \emph{Amitsur complex} $\cM\bigl(Am(X,Y)\bigr)$ is
the chain bi-complex  whose chain group in dimension $r$ is $\cM\bigl(Am_r(X,Y)\bigr)$,
with boundary operators $\bd_r = \sum (-1)^i [d^r_i]_\cM$ and
 $\delta^r= \sum (-1)^i [d^r_i]^\cM$ for $r\geq 0$. and zero
otherwise. 
\end{definition}

This construction has certain naturality properties. 
\begin{lemma} Let $\cM$ be  a Mackey functor. The Amitsur complex gives a bifunctor
$$\cM(Am(\_\, ,\_))\colon \DG \times \DG \to Chain(\ab)$$
where $Chain(\ab)$ denotes the category of chain complexes of abelian groups.
\end{lemma}
For any Mackey functor $\cM$, and any finite $G$-set $S$, let 
$\cM_S$ denote the Mackey functor defined by $\cM_S(T) = \cM(S\times T)$, for any finite $G$-set $T$. There are natural transformations  and $\Theta_S^\cM\colon \cM \to \cM_S$ and $\Theta^S_\cM\colon \cM_S \to \cM$ of Mackey functors
 induced by the  projection maps $S\times T \to T$.
 Dress says that $\cM$ is $S$-\emph{injective} (respectively $S$-\emph{projective})  if $\Theta_S^\cM$
 is split-injective (respectively $\Theta^S_\cM$ is split surjective).
 \begin{lemma}[{\cite[Prop.~1.1$^\prime$]{dress2}}]
 A Mackey functor $\cM$ is $S$-injective if and only if it is $S$-projective.
  \end{lemma}
  \begin{proof}
  Suppose that $\cM$ is $S$-projective, so that $\Theta^S_\cM$
  is split-injective. Let $\Phi\colon \cM \to \cM_S$ be a natural transformation such that $\Theta^S_\cM\circ \Phi = Id_\cM$ (the identity natural transformation on $\cM$). If $\Delta\colon S \to S \times S$
  denotes the diagonal map and $p\colon S\times T \to T$ the second factor projection, we notice that 
  $$S\times T \xrightarrow{\Delta\times 1} S\times S\times T
  \xrightarrow{1\times p} S\times T$$
  is just the identity map on $S\times T$. It follows that
  $$\Theta^S_{\cM(T)}\circ (\Delta\times 1)^\cM\circ \Phi_{S\times T}\circ \Theta_S^{\cM(T)} = Id_{\cM(T)}$$
  for any finite $G$-set $T$. One can check that the formula
  $$\widetilde\Phi(T) := \Theta^S_{\cM(T)}\circ (\Delta\times 1)^\cM\circ \Phi_{S\times T}$$
  defines a natural transformation of bi-functors splitting
$\Theta_S^\cM$ and hence $\cM$ is $S$-injective. The converse is similar.
    \end{proof}
 Dress now proves that, for any finite $G$-set $Y$,  both Amitsur complexes $(\cM_\ast(Am(S, Y), \bd)$ and
 $(\cM^\ast(Am(S, Y), \delta)$ are contractible (we say $\cM$ is $S$-\emph{computable}), whenever $\cM$ is $S$-injective or $S$-projective. In particular, for $Y = \bullet$ there are exact sequences
$$
\xymatrix@R-15pt{
\dots\ar[r]^(.3){\bd_3}&\  \cM(S\times S) \ar[r]^{\bd_2}& \cM(S) \ar[r]^{\bd_1}& \cM(\bullet)\ar[r]& 0\\
0 \ar[r]& \cM(\bullet) \ar[r]^{\delta_1}& \cM(S) \ar[r]^{\delta_2}& \cM(S\times S) \ar[r]^(.6){\delta_3}& \dots}
$$
which exhibit $\cM(\bullet)$ as a limit of induction or restriction maps respectively. 

\medskip
Here is the main theorem of Dress induction theory:

\begin{proposition}[{\cite[Prop.~1.2]{dress2}}]\label{prop: Dress main thm} Let $\cG$ be a Green ring and $S$ be a finite $G$-set. Then the following conditions are equivalent:
\begin{enumerate}
\item The map $\varphi_\cG\colon \cG(S) \to \cG(\bullet)$ associated to the projection $\varphi\colon S \to \bullet$ is surjective.
\item $\cG$ is $S$-injective.
\item All $\cG$-modules are $S$-injective.
\end{enumerate}
\end{proposition}
This result focuses attention on the task of finding a suitable Green ring which acts on $\cM$, and then checking property (i). We remark that the Burnside ring $\cA$ acts on any Mackey functor, but $\cA$ is $S$-injective only if $\bullet \subset S$. Hence the Burnside ring itself has no useful induction properties.

\section{Dress generating sets}\label{eight}
In the classical Mackey setting of $G$-functors given by Green \cite{green1}, computation is expressed in terms of families. A \emph{family of subgroups} $\cF$ of $G$ is a collection of subgroups closed under conjugation and taking subgroups. For any  finite $G$-set $X$ let $\cF(X)$ denote the family generated by the isotropy subgroups of $X$. For example, the family $\cF(\bullet) = \{\text{All}\}$.
 Conversely, given a family $\cF$ of subgroups, we can form the disjoint union $X(\cF)$ of $G$-sets $G/H$, one for each conjugacy class of maximal elements in $\cF$, under the partial ordering from subgroup inclusion. For example, $X(\{\text{All}\}) = \bullet$. We say that a family of subgroups $\cF$ \emph{contracts} a Mackey functor $\cM$ if and only if $\cM$ is $X(\cF)$-projective or $X(\cF)$-injective.

We have seen that a good strategy for computing a Mackey functor $\cM$ is to study the  Green rings acting on $\cM$. We will apply this strategy to the Burnside quotient Green ring $\cA_\cM$ of $\cM$.

\begin{definition}
Let $\cG$ be a Green ring. A finite $G$-set $X$ is a
\emph{generating set} for $\cG$ if the natural map
$\cG(X) \to \cG(\bullet)$ is surjective (or equivalently, if 
$1_\bullet \in \Image\{ \cG(X) \to \cG(\bullet)\}$).
\end{definition}
Note that by \cite[Prop.~1.2]{dress2}, $X$ is a generating set for $\cG$ if and only if $\cG$ is $X$-injective or $X$-projective. 
It is not true in general that a generating set for a Green ring $\cG$ is also a generating set for the sub-Green ring $\cA_\cG$.
To obtain generation for $\cA_\cG$ it is usually necessary to enlarge the generating set.

For $H$ a finite group and $p$ a prime, let 
$$O^p(H) = \bigcap\{H_0\triangleleft H \vv H/H_0 \text{\ is a\ } 
\text{$p$-group}\}$$
Notice that $O^p(H)$ is a characteristic subgroup of $p$-power index in $H$, and
$O^p\bigl((O^p(H)\bigr) = O^p(H)$. 
\begin{definition} Let $\cF$ be a family of subgroups of $G$ and $p$ a prime. Then
$hyper_p\text{-}\cF$ is the family consisting of all subgroups $H$ in $G$ such that $O^p(H) \in \cF$. If $S$ is a $G$-set, then $hyper_p\text{-}S$ is the corresponding $G$-set to 
$hyper_p\text{-}\cF(S)$. This construction is due to Dress \cite[p. 307]{dress2}.

\end{definition}
It is easy to check that $hyper_p\text{-}\cF$ is closed under taking subgroups and conjugation, so we obtain a family of subgroups.
By construction, there is a $G$-map $X \to hyper_p\text{-}X$ for any $X$ and $hyper_p\text{-}hyper_p\text{-}X= hyper_p\text{-}X$.
One of Dress's main results is the following:
\begin{theorem}[{\cite[p.~207]{dress1}}]\label{alternative} Let $\cM$ be a Mackey functor.
For any prime $p$, let 
 $\cK(Y)=\ker ( \cM(\bullet) \otimes \bZ_{(p)} \to \cM(Y) \otimes \bZ_{(p)} )$ and
$\cI(Y)= \Image ( \cM(hyper_p\text{-}Y)\otimes \bZ_{(p)} \to \cM(\bullet)\otimes \bZ_{(p)}
)$, for any finite $G$-set $Y$. Then $\cM(\bullet) \otimes \bZ_{(p)} = \cK(Y) + \cI(Y)$.
\end{theorem}
If $Y$ is a finite $G$-set, 
we will use the notation $\langle Y\rangle$ for the equivalence class
of $Y$ in the  category $\DG$.
One useful consequence is:
\begin{lemma}
 Let $\cG_0$ be a sub-Green ring of $\cG_1$. For any prime $p$, and any finite $G$-set
$Y$ with $\langle Y \rangle = \langle hyper_p\text{-}Y\rangle$, 
the natural map $\cG_0(Y) \otimes \bZ_{(p)} \to \cG_0(\bullet) \otimes \bZ_{(p)}$ is surjective if and only if 
$\cG_1(Y) \otimes \bZ_{(p)} \to \cG_1(\bullet) \otimes \bZ_{(p)}$
is surjective.
\end{lemma}
\begin{proof}
For any Green ring $\cG$ and any finite $G$-set $Y$, 
the image of $\cG(Y) \otimes \bZ_{(p)}$ in $ \cG(\bullet) \otimes \bZ_{(p)}$ 
is an ideal.
Hence either map is onto if and only if $1_{\cG_i(\bullet)}$
is in the image. Since $1_{\cG_0(\bullet)}$ goes to
$1_{\cG_1(\bullet)}$, this proves the first implication.

For the converse, notice that the assumption 
$\cG_1(Y) \otimes \bZ_{(p)} \to \cG_1(\bullet) \otimes \bZ_{(p)}$
is surjective implies that the Amitsur complex is contractible for the restriction maps
induced by the transformation $Y \to \bullet$. In particular,
$\cG_1(\bullet) \otimes \bZ_{(p)} \to \cG_1(Y) \otimes \bZ_{(p)}$
is injective. Therefore $\cG_0(\bullet) \otimes \bZ_{(p)} \to \cG_0(Y) \otimes \bZ_{(p)}$ is injective, and from Dress's Theorem \ref{alternative} we conclude that
$\cG_0(hyper_p\text{-}Y) \otimes \bZ_{(p)} \to \cG_0(\bullet) \otimes \bZ_{(p)}$ is surjective. 
\end{proof}

Suppose that $\cG$ is a Green ring which acts on a Mackey
functor $\cM$. 
For many applications of induction theory, 
the ``best" Green ring for $\cM$ is the \emph{Burnside quotient Green
ring} ${\cA}_{{\cG}}$.
This  is a Green
ring which acts on ${\cM}$, and by construction 
${\cA}_{{\cG}}$  is a sub-Green ring of $\cG$. In particular,
the natural map ${\cA}_{{\cG}} \to \cG$ is an injection.
\begin{definition}\label{dress_generating}
A finite $G$-set $X$ is a \emph{Dress generating
set} for a Green ring $\cG$, provided that 
$\cG(hyper_p\text{-}X) \otimes \bZ_{(p)} \to \cG(\bullet) \otimes \bZ_{(p)}$ is surjective for each prime $p$. 
\end{definition}
 By Lemma \ref{alternative}, any finite $G$-set $X$ such that the natural map $\cG(\bullet) \to \cG(X)$ is \emph{injective} is a Dress generating set for $\cG$.
Notice that a Dress generating set for $\cG$ is also a Dress generating set for any quotient Green ring of $\cG$. The following result (Theorem A) is the main step in handling sub-Green rings.
\begin{theorem}\label{main_theorem}
A finite $G$-set $X$ is a Dress generating set for a Green ring
$\cG$ if and only if it is a Dress generating set for the Burnside
quotient Green ring $\cA_{\cG}$. 
\end{theorem}
\begin{proof}
We apply the result above to $Y = hyper_p\text{-}X$, for each prime $p$, and note that $\cA_\cG$ is a sub-Green ring of $\cG$.
\end{proof}
The Burnside quotient Green ring can be used to compute Mackey functors obtained by sub-quotients.
\begin{definition}
A finite $G$-set $X$ is a \emph{Dress generating
set} for a Mackey functor  $\cM$, provided that $X$ is a Dress generating set for the Burnside quotient Green ring $\cA_\cM$ of 
$\cM$.
\end{definition}
This is consistent with our previous Definition \ref{dress_generating} for a Green ring.
\begin{theorem}\label{subquotient_properties}
 Let $\cG$ be a Green ring and $\cM$, $\cN$ Mackey functors.
\begin{enumerate}
\item If $\cM$ is a $\cG$-module and $X$ is a Dress generating set for
$\cG$, then $X$ is a Dress generating set for $\cM$.
\item If $\cN$ is a sub-quotient of $\cM$ and $X$ is a Dress generating set for
$\cM$, then $X$ is a Dress generating set for $\cN$.
\item If $\cM$ is a full lattice in $\cN$ and $X$ is a Dress generating set for
$\cM$, then $X$ is a Dress generating set for $\cN$.
\end{enumerate}

\end{theorem}

\begin{proof}
Under the first assumption, $\cA_\cM$ is a subquotient of $\cG$. In the other parts, $\cA_\cN$ is a quotient of $\cA_\cM$.

\end{proof}
We can translate this into a computability statement as follows:
\begin{corollary} Let $p$ be a prime and $\cG$ be a Green ring. Suppose that 
$\cF$ is a $hyper_p$-closed family of subgroups of $G$.  Then $\cG\otimes \bZ_{(p)}$ is $\cF$-computable
if and only if $\cA_{\cG}\otimes \bZ_{(p)}$ is $\cF$-computable.
\end{corollary}
The advantage of ${\cA}_{\cG}$ over ${\cG}$
is that ${\cA}_{\cG}$ acts on Mackey functors
which are subfunctors or quotient functors of ${\cM}$
but ${\cG}$ does not in general.
For example, ${\cG}$ never acts on ${\cA}_{\cG}$
unless they are equal. 
We next point out another good  feature
 of the Burnside quotient Green ring.

\begin{theorem}[{\cite[Theorem~1.8]{h4}}]\label{dress-refinement}
 Suppose that $\cG$ is a Green ring which
acts on a Mackey functor $\cM$,  and 
$\cF$ is a $hyper_p$-closed family of subgroups of $G$. If $\cG\otimes \bZ_{(p)}$ is $\cF$-computable, then  every
$x\in \cM(G)\otimes \bZ_{(p)}$ can be written as
$$x = \sum_{H \in \cF} a_H \Ind_H^G(\Res_G^H(x))$$
for some coefficients $a_H \in  \bZ_{(p)}$, where the $a_H$ are the
same for all $x$.
\end{theorem}
\begin{proof} 
Since $\cG\otimes \bZ_{(p)}$ is $\cF$-computable, we know that $\cA_{\cG}\otimes \bZ_{(p)}$ is
also $\cF$-computable. Therefore, we can write 
$1 = \sum_{K\in \cF} b_K \Ind_K^G(y_K)$, for some $y_K \in 
\cA_{\cG}(K)\otimes \bZ_{(p)}$ and $b_K\in\bZ_{(p)}$. For  any $x
\in \cM(G)\otimes \bZ_{(p)}$ we now have the formula
$$x = 1\cdot x = 
\sum_{K\in \cF} b_K \Ind_K^G(y_K\cdot \Res_G^K(x))$$
But each $y_K \in \cA_{\cG}(K)\otimes \bZ_{(p)}$ can be represented by 
a sum $\sum c_{KH}[K/H]$, with $c_{KH}\in  \bZ_{(p)}$, under the surjection $\cA(K) \to \cA_{\cG}(K)$. Therefore
$$
\begin{array}{ll}
x & = \sum_{K\in \cF} b_K
\sum_{H \subseteq K} c_{KH} \Ind_K^G([K/H]\cdot \Res_G^K(x))\\
&\\
& = \sum_{K\in \cF} b_K\sum_{H \subseteq K} c_{KH}
\Ind_K^G(\Ind_H^K(\Res_K^H( \Res_G^K(x))))\\
&\\
& =  \sum_{K\in \cF} b_K\sum_{H \subseteq K} c_{KH}
\Ind_H^G( \Res_G^H(x))
\end{array}
$$
We now define $a_H = \sum_{K \in \cF} b_K c_{KH}$,
and the formula becomes
$$x = \sum_{H\in \cF} a_H \Ind_H^G( \Res_G^H(x))\ .$$
\end{proof}
\begin{example}[Representation theory] 
Recall that a $p$-(hyper)elementary group is a (semi)direct product $C \rtimes P$, where $P$ is a $p$-group and $C$ is cyclic of order prime to $p$. 
A Dress generating set for a Green ring $\cG$ need not be a 
generating set for $\cG$. For example,  let $E$ denote the finite $G${-}set, $E = \coprod G/H$,
where we have one $H$ for each $p${-}elementary subgroup of $G$.
It is known that $E$ is a generating set for the complex representation ring $R_\bC(G)$, but not in general for  
the  rational representation ring $R_\bQ(G)$. 
On the other hand, complex representations are detected by characters, so any $G$-set with isotropy containing the cyclic family is a Dress generating set for 
$R_\bC(G)$, or for the sub-Green ring  $R_\bQ(G)$ by Theorem \ref{subquotient_properties} (ii).
It follows that 
 the hyperelementary family $\cH$ gives a generating set $X_{\cH}$  for  $R_\bQ(G)$. This implies the Brauer-Berman-Witt induction theorem for rational representations.
\end{example}
\begin{example}[The Swan ring]\label{swan_ring}
 The Swan ring is one of the main examples of Green rings in the classical setting of induction theory (see \cite{swan2}). For any finite group, let $SW(G, \bZ)$ denote the Grothendieck group of isomorphism classes of finitely-generated left $\bZ G$-modules, with $[L] = [L^\prime] + [L^{\prime\prime}]$ whenever there is a short exact sequence
$$0 \to L^\prime \to L \to L^{\prime\prime}\to 0$$
of such $\bZ G$-modules. The operation $L\otimes_{\bZ} L^\prime$ gives a ring structure on this Grothendieck group, so we obtain a commutative ring.  The usual induction and restriction operations for such modules give the Swan ring the structure of a Mackey functor. We let
$$SW_G\colon \cD(G) \to \Ab$$
denote the Green ring (in the sense of Dress) defined by 
$SW_G(G/H) := SW(H, \bZ)$, and extended to $\cD(G)$ by additivity.  
Since $SW(G,\bZ)$ is hyperelementary computable by Swan's induction theorem (see \cite[p.~211]{dress1}), we see that any Mackey functor on which this Green ring acts is hyperelementary computable. 

It follows that the Burnside quotient Green ring
of the Swan ring, denoted $\cA_{SW}$, also has the hyperelementary set $X_{\cH}$ as a Dress generating set (or more precisely, any $G$-set whose isotropy contains the cyclic family is a Dress generating set). In this case, $\cA_{SW}(G/H) \subset SW_G(G/H)$ is the subring  $P(H, \bZ) \subset SW(H,\bZ)$ generated by the permutation modules $\bZ[H/K]$, for all subgroups $K \subseteq H$.
\end{example}
\section{Computation techniques}\label{two}
Dress generating sets can also be used to compute exact sequences of Mackey functors or filtrations of Mackey functors by sub-functors. We say that 
$$\cM_0 \xrightarrow{a} \cM_1 \xrightarrow{b} \cM_2$$
is an \emph{exact sequence} of Mackey functors if $a$ and $b$ are 
homomorphisms of Mackey functors, such that the sequence $\cM_0(S) \to \cM_1(S) \to \cM_2(S)$ is exact for each finite $G$-set $S$. We define long exact sequences in a similar way.

\begin{proposition}\label{short_exact}
 Suppose that $\cM_0\to \cM_1\to
\cM_2$ is an exact sequence of Mackey functors. Then $X$ is a Dress generating set for $\cM_1$ whenever $X$ is a Dress generating set for $\cM_0$ and   $\cM_2$.
\end{proposition}
\begin{proof} We may assume that $\cM_0\to \cM_1$ is injective, and that  $\cM_1\to \cM_2$  
is surjective, and the projections from $\cA$ induce a natural transformation $\theta\colon \cA_{\cM_1} \to \cA_{\cM_0}\oplus \cA_{\cM_2}$ of Green rings. 
By exactness, $I_\theta(S):=\ker\theta_S$ is a nilpotent ideal (of nilpotence index 2). Let $\bar\cA_{\cM_1}= \cA_{\cM_1}/I_\theta$
denote the quotient Mackey functor. Since this Mackey functor is
mapped injectively by $\theta$ into $\cA_{\cM_0}\oplus \cA_{\cM_2}$,
we see that $X$ is a Dress generating set for $\bar\cA_{\cM_1}$.
It follows that $\bar\cA_{\cM_1}(hyper_p\text{-}X)\otimes \bZ_{(p)}\to \bar\cA_{\cM_1}(\bullet)\otimes \bZ_{(p)}$ is surjective  for every prime $p$.
But an element in $\cA_{\cM_1}(\bullet)\otimes \bZ_{(p)}$
 hitting $1_\bullet\in \bar\cA(\bullet)\otimes \bZ_{(p)}$ has the form $1 + u$, where
 $u \in I_\theta(\bullet)\otimes \bZ_{(p)}$. 
  Since
$u^2=0$, $1 + u$  is invertible and $(1+u)^{-1} = 1-u$. If $p_X\colon X \to \bullet$ denotes the projection map, and  $(p_X)_\ast(a) = 1+u$, then we have
$$(p_X)_\ast\bigl ((p_X)^\ast(1_\bullet-u)\cdot a\bigr ) = (1_\bullet - u)\cdot (p_X)_\ast(a) = 1_\bullet$$
and hence $X$ is a Dress generating set for $\cA_{\cM_1}$.
\end{proof}
\begin{remark}\label{lifting}  
In the proof of Proposition \ref{short_exact}, we have shown  for each prime $p$, there exists an element $a\in \cA(hyper_p\text{-}X)\otimes \bZ_{(p)}$ such that $a\mapsto 1_\bullet$ in each of the Burnside quotient Green rings
$\cA_{\cM_i}(\bullet)\otimes \bZ_{(p)}$, for $i = 0, 1, 2$. 
The same argument extends by induction to finite filtrations of a Mackey functor by sub-Mackey functors.
\end{remark}
\begin{corollary}\label{finite_filtration}
Let $0 = \cN_0\subset \dots \subset \cN_r = \cM$ be a filtration of a Mackey functor by  sub-Mackey functors. Then $X$ is a Dress generating set for $\cM$ if and only if $X$ is a Dress generating set for each quotient 
$\cN_i/\cN_{i-1}$, for $1\leq i \leq r$. 
\end{corollary}
A finite length \emph{chain complex} of Mackey functors is a sequence $(\cN_i, \bd_i)$ of Mackey functors $\cN_i$, $0 \leq i \leq r$, and natural transformations $\bd_i\colon \cN_i \to \cN_{i-1}$, $1\leq i \leq r$, such that 
$(\cN(S), \bd_\ast)$ is a  chain complexes of abelian groups for each finite $G$-set $S$.
 A chain complex $\cN$ of
Mackey functors has homology groups $H_i(\cN)$, $0\leq i \leq r$, which are subquotient Mackey functors of $\cN_i$.
\begin{corollary}   Suppose that $\cN$ is a finite length chain complex of Mackey functors.  If $X$ is a Dress generating set for each $\cN_i$, $0\leq i \leq r$,  then $X$ is a Dress generating set for each of the homology Mackey functors $H_i(\cN)$, $0\leq i \leq r$.
\end{corollary}
Another useful construction is ``completion".  
\begin{theorem}\label{filtration_quotients}  
Let $\cM$ be a Mackey functor, and let $\cF$ denote  a (possibly infinite) filtration
$$\cM = F_0 \supseteq F_1 \supseteq \dots \supseteq F_r \supseteq \dots$$
of $\cM$ by sub-Mackey functors. 
 A finite $G$-set $X$ is a Dress generating set for
$\widehat \cM_\cF= \invlim \cM/F_r$ if and only if $X$ is a Dress generating set for
each quotient Mackey functor $F_{r-1}/F_r$, $r\geq 1$.
\end{theorem}
\begin{proof}   Since each $F_{r-1}/F_r$ is a sub-quotient of $\widehat \cM_\cF$, the necessity follows from the results above. For sufficiency, 
we first note by Corollary \ref{finite_filtration} that $X$ is Dress generating set for each quotient $\cM/F_r$. It is enough to prove that $X$ generates the inverse limit
$\invlim \cA_{\cM/F_r}$ of the Burnside quotient Green rings for the sequence $\{\cM/F_r\}$.
Suppose that $X$ is a Dress generating set for each 
$\cA_{\cM/F_r}$, $r\geq 1$, and set $Y =hyper_p\text{-}X$. If $\{a_r\}$ is a sequence of elements in
$\cA_{\cM/F_r}(Y)$ hitting $1^r_\bullet$, we can use the contractibility of the $Y$-Amitsur complex for  $\cA_{\cM/F_{r+1}}$  inductively, to adjust each $a_{r+1}$ by an element of $\cA_{\cM/F_{r+1}}(Y \times Y)$, so that $a_{r+1} \mapsto a_r$. This gives us an element in the inverse limit $\invlim\cA_{\cM/F_r}(Y)$ hitting $1_\bullet\in
\invlim\cA_{\cM/F_r}(\bullet):=\cG(\bullet)$, and hence the Green ring $\cG$ acts on 
$\widehat\cM_\cF$ with $X$ as a Dress generating set. Since $\cA_\cG \twoheadrightarrow \cA_{\widehat\cM_\cF}$ is surjective, it follows that  $\widehat\cM_\cF$ has $X$ as a Dress generating set. 
\end{proof}
\begin{example}   Here is an important special case. Let $\cG$ be a Green ring acting on a Mackey functor $\cM$. If $\cI\subset \cG$ is a Green ideal, we may filter $\cM$ by the sub-bifunctors $F_r=\cI^r\cM$ and then  $\widehat\cM_\cF$ is the $\cI$-adic completion of $\cM$.

In particular, for a given Mackey functor $\cM$ we could take
$\cI = \langle \cI_\cM, 2\rangle$, and then $\widehat \cM_\cI$ is just the $2$-adic completion of the Mackey functor $\cM$.  Note that if $\cM(X)$ is  finitely generated, then $\widehat \cM_\cI(X) \cong \cM(X)\otimes \Zhat_2$.  
\end{example}

%%%%%%%%%%%%%%%%%%%%%%%%%%%%%%%
\section{Mackey functors and $RG$-Morita  }\label{three}
In order to prove Theorem B, we need to define the bifunctor $d\colon \cD(G) \to \RGMorita$ used in its statement. This involves some definitions and elementary properties of categories with bisets as morphisms, which are  well-known to the experts. We include this material for the reader's convenience.

In \cite[1.A.4]{htw3} we introduced the category $RG$-Morita whose basic objects are finite groups $H$ isomorphic to some subquotient of $G$, and whose morphisms were defined by a Grothendieck group construction on the isomorphism classes of finite $H_2$-$H_1$ bisets $X$, for which the order of the left stabilizer
$$ _{H_2}I(x) = \{ h \in H_2\vv hx=x\}$$
is a unit in $R$, for all $x \in X$. 
Here $R$ is a commutative ring with unit. 
We set $X \sim X^\prime$ 
if $RX$ is isomorphic to $RX^\prime$ as $RH_2$-$RH_1$ bimodules.
The balanced product $ X\times_{H_2} Y$ of an 
$H_3$-$H_2$ biset $X$ and an $H_2$-$H_1$ biset $Y$ is a  $H_3$-$H_1$ biset. This defines the composition for morphisms. The Add-construction \cite[p.~194]{maclane2} is then applied to complete the definition.
Many functors arising in algebraic $K$-theory and topology are actually functors out of $RG$-Morita, so it is of interest to recognize when these are Mackey functors.

To relate Mackey functors and $\RGMorita$, will need the 
$G$-Burnside category, $\Burnside(G)$,  whose
objects are subgroups $H\subset G$, and where
$\Hom_{\Burnside(G)}(H_1, H_2)$ is the Grothendieck construction applied to
the isomorphism classes of finite bifree $H_2$\,-$H_1$ bisets
(meaning both left and right actions are free). 
Because of the Grothendieck group construction, $\Burnside(G)$ is an  Ab--category: the morphism sets are
abelian groups and the compositions are bilinear  \cite[I.8, p.~28]{maclane2}.
Let $$u\colon \bA(G) \to \bAdot(G)$$  denote the associated universal free additive category, and  the universal inclusion (see  \cite[VII.2, problem 6, p.~194]{maclane2}). 

The morphisms in $\Burnside(G)$ are defined by  the Grothendieck group construction with
addition operation the disjoint union of bisets. By convention, the empty biset  $\emptyset$ represents the zero element.
Composition comes from the balanced product:
$${_{H_3}}X_{H_2} \circ {_{H_2}}X_{H_1} = ({_{H_3}}X_{H_2})\times_{H_2} ({_{H_2}}X_{H_1}).$$
The reader should check that this is well--defined on isomorphism classes of bisets
and ``bilinear'' in that
$$({_{H_3}}X_{H_2} \disjointunion {_{H_3}}Y_{H_2}) \circ {_{H_2}}X_{H_1}\cong ({_{H_3}}X_{H_2} \circ{_{H_2}}X_{H_1})\disjointunion\,
({_{H_3}}Y_{H_2}\circ {_{H_2}}X_{H_1})$$ with a similar formula for disjoint union on the right. The morphisms in $\bAdot(G)$ are matrices of morphisms in $\Burnside(G)$.

\begin{definition}
We define a contravariant involution $\congtran\colon \Burnside(G) \to \Burnside(G)$,  by the identity on objects,
and on morphisms it is the map induced on the Grothendieck construction by the function
which takes the finite bifree $H_2$\,-$H_1$ biset ${_{H_2}}X_{H_1}$ to the  finite bifree
$H_1$\,-$H_2$ biset ${_{H_1}}X_{H_2}$ which is $X$ as a set and
$h_1\cdot x \cdot h_2$ is defined to be $h_2^{-1}x h_1^{-1}$.
\end{definition}

The reader needs to check that isomorphic bisets are isomorphic after reversing the order,
and should also check that the transpose conjugate of a disjoint union is isomorphic to the
disjoint union of the conjugate transposes of the pieces.
This means that $\congtran$ is a functor which induces a homomorphism of Hom--sets.
It is clearly an involution, not just up to natural equivalence.
Since $\congtran$ is a homomorphism on Hom-sets, it induces an additive contravariant
involution $\congtrandot\colon \bAdot(G) \to \bAdot(G)$, called \emph{conjugate transpose}, which commutes with the functor $u$.  By definition,  $\congtrandot$ acts on a matrix of morphisms by applying $\congtran$ to each entry, and then transposing the matrix.
There is a functor $$a\colon \bAdot(G) \to \RGMorita$$ given by the inclusion on objects and morphisms (but the equivalence relation on morphisms is different in $\RGMorita$).

There is a functor  $\bAdot(G) \to \RMorita$, called the $R$-\emph{group ring functor}, where $R$-Morita has objects $R$-algebras and morphisms defined by stable isomorphism classes of bimodules (see \cite[1.A.1]{htw3}). This functor factors through $\RGMorita$: it sends $H \mapsto RH$ on objects, and $X \mapsto RX$ on morphisms. 

We will define the following diagram of categories and functors:
\eqncount
\begin{equation}\label{main_diagram}
\vcenter{\xymatrix@C+4pt{&\bAdot(G) \ar[d]^a\ar[dr]^{\ \ R\text{-group\,ring}}&\cr
\cD(G) \ar[r]^(0.4)d\ar[ur]^(0.33)j & \RGMorita\ar[r]& \RMorita}}\end{equation}

To complete the definition of the functors in this diagram, we need to introduce another category.
Let $\cDdot(G)$ denote the category whose objects are pairs $(X, \mathbf b)$,
consisting of a finite $G$-space $X$ and an ordered collection $\mathbf b = (b_1, \cdots, b_n)$ of base-points,
one for each $G$-orbit of $X$.
The morphisms are the $G$-maps (not necessarily base-point preserving).
There is a functor
$$\DdotD\colon \cDdot(G) \to \cD(G)$$
defined by forgetting the base-points.
Since every object of $\cD(G)$ is isomorphic to the image $\DdotD(X,\mathbf b)$ of an object of
$\cDdot(G)$, and $\DdotD$ induces a bijection on morphism sets, it follows that
$\DdotD$ gives an equivalence between the categories
$\cDdot(G)$ and $\cD(G)$, with inverse functor $\DtoDdot$
\cite[IV.4, Theorem 1, p.~91]{maclane2}.

We can now define two functors 
$$(\downj, \upj)\colon \cDdot(G) \to \bAdot(G).$$
The covariant functor $\downj$ is the additive extension of the functor which sends an object $(G/H, eH)$ to the  isotropy subgroup $H$, and sends the $G$ map $f\colon G/H \to G/K$
to the biset ${_{K}}K_{{g^{-1} H g}}$, where $f(e H) = g K$. If we change the coset representative and write $f(eH) = g_1K$, then the map
\eqncount
\begin{equation}\label{bisetbijection}
\psi\colon {_{K}}K_{{g^{-1} H g}} \rightarrow {_{K}}K_{{g_1^{-1} H g_1}}
\end{equation}
defined by $\psi(k) = k(g^{-1}g_1)$ gives an bijection of $K$-$K$ bisets.
Note that $1_{G/H}\colon G/H \to G/H$ goes to ${_{H}}H{_{H}}$ which is the identity.
Check that if
$f_1\colon G/H_1\to G/H_2$ and $f_2\colon G/H_2 \to G/H_3$ are $G$-maps and if
$f_1(e H_1) = g_1 H_2$ and $f_2(e H_2) = g_2 H_3$
then $f_2\circ f_1(e H_1) = (g_1 g_2) H_3$ and
$${_{H_3}}\bigl(H_3\bigr){_{g^{-1}_2 H_2 g_2}} \times_{H_2}
{_{H_2}}\bigl(H_2\bigr){_{g^{-1}_1 H_1 g_1}}$$
is isomorphic to ${_{H_3}}\bigl(H_3\bigr){_{(g_1 g_2)^{-1} H_1(g_1 g_2)}}$ by the map
$(h_3, h_2) \mapsto h_3 g_2^{-1} h_2 g_2$.

The contravariant functor $\upj$ agrees with $\downj$ on objects, but sends the $G$ map $f\colon G/H \to G/K$
to the  biset ${_{g^{-1} H g}}K_{K}$ where $f(e H) = g K$. 
Rather than check identity and composition directly, just note that
${_{g^{-1} H g}}K_{K}$ is isomorphic to $\congtran\bigl({_{K}}K_{{g^{-1} H g}}\bigr)$ by the
function which sends $k$ to $k^{-1}$,
so $\upj = \congtran \circ \downj$ and hence $\upj$ is a contravariant functor.

\begin{definition}\label{DtoBurnside}
We define the bifunctor
$$\DtoAdot\colon \cD(G)\to \bAdot(G)$$
as the composition $\DtoAdot = (\upj, \downj)\circ \DtoDdot$.  Let
$$d= a\circ j\colon \cD(G) \to \RGMorita$$
denote the composition in diagram (\ref{main_diagram}).
\qed
\end{definition}

For any additive functor $F\colon \bAdot(G) \to \Ab$, the composition  $F\circ j\colon \cD(G) \to \Ab$ is a Mackey functor (see \cite{htw2008}). Our main application is the following:
\begin{theorem}\label{rgmorita}
 Any additive functor $F\colon \RGMorita\to \Ab$ gives a Mackey functor on $\cD(G)$ by composition with $d\colon \cD(G) \to \RGMorita$. Any such Mackey functor is hyperelementary computable.
\end{theorem}
\begin{proof}
The  functor $d\colon \cD(G) \to \RGMorita$  factors through $\bAdot(G)$, so we obtain Mackey functors by composition. We will show that any such  Mackey functor, $\cM$, is a Green module over the Burnside quotient Green ring $\cA_{SW}$ of the Swan ring,  and then apply Example \ref{swan_ring}. Let $L = \bZ [H/K]$ denote a permutation module, for some subgroups $K \subset H$ of $G$, and let $X$ denote an $H$-$H$-biset, which is free as a left $H$-set. Then $H/K \times X$ is again an $H$-$H$-biset by the formula $h_1(hK, x)h_2 = (h_1hK, h_1xh_2)$, for all $h, h_1, h_2 \in H$ and all $x\in X$. Note that $R[H/K \times X] = L\otimes_\bZ RX$  as $RH$-$RH$ bimodules, so this
 construction applied to $X = {}_HH_H$, sending $\bZ[H/K] \mapsto H/K \times H$, gives a well-defined homomorphism
$$P(H, \bZ) \to \Hom_{\RGMorita}(H,H)$$
from the Grothendieck group of permutation modules, for each subgroup $H \subset G$. The adjoints of these homomorphisms give a pairing $\cA_{SW} \times \cM \to \cM$, and the Green module properties follow easily from bimodule identities (compare \cite[11.2]{oliver3}).
Since $\cA_{SW}$ is 
 hyperelementary computable,  we conclude that any Mackey functor out of $\RMorita$ is hyperelementary computable.
 \end{proof}
 \begin{remark} As mentioned in the Introduction, this is a refinement of an earlier result of Oliver \cite[11.2]{oliver3}. Oliver establishes hyperelementary computability for functors of the form $X(R[G])$, where $X$ is an additive functor from the category of $R$-orders in semisimple $K$-algebras with bimodule morphisms to the category of abelian groups. Here $R$ is a Dedekind domain with quotient field $K$ of characteristic zero. 
 
 There are two points of comparison: it should first be noted that Oliver \cite[p.~246]{oliver3} is dealing with Mackey functors defined on the category of finite groups and monomorphisms, so the statement that any such functor $X(R[G])$ is a Mackey functor is straight-forward. In our case, relating $RG$-Morita to Mackey functors defined on finite $G$-sets in the sense of Dress \cite[p.~301]{dress2} involves some work (e.g.~in constructing the bifunctor $d$). The translation between the two versions of Mackey functors is also well-known to the experts (see \cite[Section 1]{dress2}), but in this paper we preferred to work only with the Dress $G$-set theory. 
 
 The second point of comparison is that Oliver's proof uses an action of the Swan ring on the Mackey functors $X(R[G])$, but the Swan ring does not act on our functors in any obvious way. The key new ingredient in our proof is the Burnside quotient Green ring of the Swan ring. Apart from this additional input, the argument is essentially the same. However, the extra generality can be useful since there are functors out of $RG$-Morita which do not appear to extend to the setting of \cite[11.2]{oliver3}.  \qed
  \end{remark}
 \begin{example}[Controlled topology] The bounded categories $\cC_{M, G}(R)$ of \cite[\S 4]{hp1}, and the continuously controlled categories $\cB_G(X \times [0,1);R)$ of
 \cite[\S 6]{hp4} are  functors out of $\bAdot(G)$, for any finite group $G$, and hence any additive functor from these categories to abelian groups gives a Mackey functor on $\cD(G)$. 
\end{example}
\begin{example}[Farrell-Hsiang induction] There is a useful extension of induction theory to (possibly) infinite groups, due to Farrell and Hsiang \cite{farrell-hsiang2}. Given any representation $\pr\colon \Gamma \to G$, with $G$ finite, we get a new $R$-group ring functor $\cA(G) \to \RMorita$ by sending $G/H \mapsto R[\Gamma_H]$, where $\Gamma_H = \pr^{-1}(H)$ is the pre-image of $H$ in $\Gamma$. We have a generating set for the morphisms $\Hom_{\bA(G)}(H_1, H_2)$  consisting of the bisets $H_2 \times_K H_1$, where $K \subset H_2\times H_1$ is a subgroup (see \cite[1.A.9]{htw3}). We send the biset $H_2 \times_K H_1$ to the bimodule $R[\Gamma_{H_2}]\otimes_{R[\Gamma_{K}]} R[\Gamma_{H_1}]$. By composition with any additive functor $F\colon \RMorita \to \Ab$, we again obtain Mackey functors. Since the Swan ring acts on $\RMorita$ (by tensor product as above), any such Mackey functor is a a Green module over the Swan ring, and we obtain hyperelementary computation as before. The main examples are listed in \cite[1.A.12]{htw3}, including Quillen $K$-theory $K_n(R[\Gamma])$. 
\end{example}
\begin{remark}
An alternate (and slightly sharper) formulation of this example could be given by defining $\RGammaMorita$ for any discrete group $\Gamma$: the objects are finite groups $H$ isomorphic to some subquotient $H\cong \Gamma_1/\Gamma_0$ of $\Gamma$, where $\Gamma_0 \triangleleft \Gamma_1$ and $\Gamma_1$ is finite index in $\Gamma$. The morphisms are $H_2$-$H_1$ bisets as before. Then from any representation $\pr\colon \Gamma \to G$, where $G$ is finite, we get a functor $d\colon \cD(G) \to \RGammaMorita$ and  Theorem \ref{rgmorita}  holds in this new setting.
\end{remark}
\begin{example}[Cohomotopy]\label{cohomotopy}
%The bifunctor which assigns to $H\subset G$ the $0^{\hbox{th}}$ 
%stable cohomotopy group $\pi ^0(BH)$ is a Green ring, which is projective with respect to the 
%which is easily checked to have the Sylow family as a contracting family.
%Hence any Mackey functor derived from the cohomology of a space and its
%finite coverings
% has the Sylow family as a contracting family.
%More explicitly, let $X, x_0$ be a space (semi{-}locally 1{-}connected)
%and a representation $\pi_1(X,x_0)\to G$, define $X(H)$ for any $H\subset G$
%as the corresponding cover. 
%Given any homology theory $E_{\ast}$ or cohomology theory, $E^{\ast}$,
%the assignment $H\mapsto E_{\ast}\bigl(X(H)\bigr)$ or 
%$H\mapsto E^{\ast}\bigl(X(H)\bigr)$ are Mackey functors which are modules 
%over the Green ring $H\mapsto \pi^0(BH)$.
Lam, \cite[Ex.~2.1 and Remark 2.4, p\prd108-109]{lam1}, 
shows that cohomology (ordinary or Tate) with twisted coefficients
$H^i(\qm;M)$ is a Mackey functor on $\cD(G)$ where $M$ 
is a fixed $G${-}module. 
Since the cohomotopy Green ring $H \mapsto \pi ^0(BH)$ acts on this 
Mackey functor, it is Sylow computable. 
If $\pr\colon\Gamma \to G$ is a homomorphism and $M$ is a $\Gamma${-}module,
then $H^i(\pr^{-1}(\qm);M)$ is also a Mackey functor on $\cD(G)$ with the 
cohomotopy Green ring acting. 
An interesting example of this situation is Galois cohomology.
\end{example}

\section{Pseudo-Mackey functors and pseudo-complexes}\label{four}
 We wish to apply the computation strategy described above to a more general situation, namely to study functors which have induction and restriction but are not known to be Mackey.  The main examples of interest are the higher Whitehead groups $\wh_n(\bZ G)$ and the non-oriented surgery obstruction groups $L_n(\bZ G, \omega)$.
\begin{definition} A covariant \emph{pre-functor} $f\colon \cD \to \cE$ between two categories is just a function $S \mapsto ob(f)(S)$ on objects, and a function $$hom(f) \colon \Hom_\cD(S_1,S_2) \to \Hom_\cE(ob(f)(S_1), ob(f)(S_2))$$ on Hom-sets. A functor is a pre-functor which preserves identities and compositions. Similarly, we define a contravariant pre-functor, and a \emph{pre-bifunctor} then consists of a pair
$(f_\ast, f^\ast)$ of pre-functors, where $f_\ast$ is covariant, $f^\ast$ is contravariant, and $ob(f_\ast) = ob(f^\ast)$. We call these \emph{Mackey pre-functors} if $\cD = \DG$ and $\cE = \ab$.

A  \emph{pre-natural transformation} $T\colon f_1\to f_2$ is a function 
$$S\mapsto T(S) \in \Hom_\cE(ob(f_1)(S), ob(f_2)(S))\ .$$
 A \emph{natural transformation} of (covariant) pre-functors is a pre-natural transformation \newline
$T\colon f_1\to f_2$ such that the diagram
$$\xymatrix{ob(f_1)(S_1)\quad \ar[r]^{hom(f_1)(\phi)}\ar[d]_{T(S_1)}&\quad
ob(f_1)(S_2) \ar[d]^{T(S_2)}\\
ob(f_2)(S_1)\quad \ar[r]^{hom(f_2)(\phi)}&\quad
ob(f_2)(S_2)\\}$$
commutes for all pairs of objects $S_1, S_2\in\cD$ and all 
$\phi\in \Hom_\cD(S_1,S_2)$.
There is a similar definition for (pre-)natural transformations of contravariant pre-functors, and a natural transformation of pre-bifunctors is a single function which is natural transformation for both the covariant and contravariant parts of the bifunctor.   A \emph{pre-pairing} between three Mackey pre-functors $\cM$, $\cN$ and $\cL$ is a collection of functions
$\mu(S) \colon \cM(S) \times \cN(S) \to \cL(S)$. Finally, if $\cM \to \cN$ is an injective natural transformation of Mackey pre-functors, then we say that $\cM$ is a sub Mackey pre-functor of $\cN$.
\end{definition}
Note that if $\cM\colon \DG \to \ab$ is a Mackey pre-functor, we can apply $\cM$ to any of the Amitsur complexes $Am(X,Y)$, and obtain
$\bd_r$ and $\delta^r$ maps as usual, but we can not be sure that
$\bd_r\circ \bd_{r+1}=0$ or $\delta^{r+1}\circ\delta^{r} =0$. We call 
$\cM(Am(X,Y))$ a \emph{pre-Amitsur} complex. This construction gives a pre-functor $\DG \times \DG \to Chain(\ab)$.
\begin{definition} A Mackey pre-functor $\cM$ is called a \emph{pseudo-Mackey functor} provided that there exists a finite collection of Mackey pre-functors  $0=\cN_0\subset \cN_1\subset\dots \subset \cN_r=\cM$ such that the quotient pre-bifunctors $\cN_i/\cN_{i-1}$  are actually  Mackey functors, for $1\leq i \leq r$. The collection $\{\cN_i/\cN_{i-1}\vv 1\leq i \leq r\}$ will be called the \emph{associated graded} Mackey functor to $\cM$.
\end{definition}
A natural transformation $\cM \to \cN$ of pseudo-Mackey functors is a natural transformation of Mackey pre-functors which preserves the filtrations.
Notice that the Burnside ring $\cA$ ``acts" on a Mackey pre-functor via the usual formula (which gives a pre-pairing). The action of $\cA$ on a pseudo-Mackey functor $\cM$ preserves the filtration, and the induced action on the sub-quotients $\cN_i/\cN_{i-1}$ is the usual action.

We say that a finite $G$-set $X$ is a Dress generating set for a pseudo-Mackey functor $\cM$ provided $X$ is a Dress generating set for each of the Mackey functors $\cN_i/\cN_{i-1}$ in its associated filtration. This agrees with our previous definitions if $\cM$ is a Mackey functor filtered by Mackey sub-functors.   Notice that the image of the natural map of Green rings $\cA \to \bigoplus_{i=1}^r \cA_{\cN_i/\cN_{i-1}}$ is a Green ring with $X$ as a Dress generating set. It follows that there exists an element $a \in \cA(hyper_p\text{-}X)\otimes \bZ_{(p)}$, for each prime $p$, whose image
in $\cA(\bullet)\otimes \bZ_{(p)}$ acts as $1_\bullet$ on each
sub-quotient $\cN_i/\cN_{i-1}(\bullet) \otimes \bZ_{(p)}$, $1\leq i \leq r$.
\begin{lemma}\label{pseudo_one}
Suppose that $\cM_0\to \cM_1$ and $\cM_1\to \cM_2$ are  natural transformations of Mackey pre-functors, such that $\cM_0(Y) \to \cM_1(Y) \to \cM_2(Y)$ is exact for every finite $G$-set $Y$.
If $\cM_0$ and $\cM_2$ are pseudo-Mackey functors, then   $\cM_1$ is a pseudo-Mackey functor. Moreover, if $X$ is a Dress generating set for $\cM_0$ and $\cM_2$, then $X$ is a Dress generating set for $\cM_1$.
\end{lemma}
\begin{proof}
The pre-image of the associated filtration for $\cM_2$ gives a filtration 
$\cN_0\subset \cN_1\subset\dots \subset \cN_r=\cM_1$, with $\cM_0\subset \cN_i$ for $0\leq i \leq r$. Since a sub-bifunctor of a Mackey functor is Mackey, we see that the quotient pre-functors
$\cN_i/\cN_{i-1}$ are actually Mackey functors (and they all have Dress generating set $X$ by   Theorem \ref{subquotient_properties}). Now we extend this filtration by adjoining the associated filtration for $\cM_0$. Since each of the sub-quotients in this extended filtration have Dress generating set $X$, the result follows.
\end{proof}

We also get a computational result for pseudo-Mackey functors. The Amitsur pre-complex $(\cM_\ast(Am(X,Y), \bd_\ast)$ is now a pseudo-complex, meaning that the boundary maps $\bd_\ast$ are filtration-preserving (and the associated graded is an actual complex). It will be called \emph{pseudo-contractible} if it is equipped with degree +1 filtration-preserving natural  transformations  $$s_r\colon \cM(Am_r(X,Y)) \to \cM(Am_{r+1}(X,Y))$$ of  pre-functors, for $r\geq 0$, which contract  the Amitsur complexes for the associated graded Mackey functors to $\cM$.
The collection $s_\ast=\{s_r\}$ is called a \emph{pseudo-contraction}. We make a similar definition for the cochain Amitsur complex and the degree -1 cochain pseudo-contractions $\sigma^r$.

We can construct pseudo-contractions by using any element $a \in \cA(X)$ such that $a$ acts as $1_\bullet \in $  on each
sub-quotient $\cN_i/\cN_{i-1}(\bullet) \otimes \bZ_{(p)}$, $1\leq i \leq r$, to build
``chain homotopies" $s_r(a)$ and ``cochain homotopies" $\sigma^r(a)$. These are pseudo-contractions in the above sense.

\begin{proposition}\label{pseudo_two} Let $\cM$ be a pseudo-Mackey functor, and $X$, $Y$ finite $G$-sets. If  $(\cM_\ast(Am(X,Y), \bd_\ast)$ is  pseudo-contractible with pseudo-contraction $s_\ast$, then there are canonical filtration-preserving natural 
transformations  $(\bd^\prime_\ast, s^\prime_\ast)$  for which
$(\cM_\ast(Am(X,Y), \bd^\prime_\ast)$ is a chain complex, and $s^\prime_\ast$ is a chain contraction.  If the pseudo-complex was already a complex, $\bd^\prime_\ast = \bd_\ast$, and if in addition $s_\ast$ was already a contraction, then $s^\prime_\ast=s_\ast$. 
\end{proposition}
\begin{proof}
Let $(C_i,\bd_i,s_i)$ be our data, where $\bd_i$ and $s_i$ are natural transformations. We assume that for $i<r$,
$\bd_i\circ\bd_{i+1} =0$, $\bd_{i+1}= \bd_{i+1}\circ s_i\circ\bd_{i+1}$,
and $s_{i-1}\circ \bd_i + \bd_{i+1}\circ s_i = 1_{C_i}$. For $r\leq 0$ these identities clearly hold. We proceed to show how these conditions may be achieved for $i=r$ by modifying $\bd_{r+1}$ and $s_r$ (if necessary).  Throughout the inductive construction, we do not change the maps induced by $(\bd_\ast, s_\ast)$ on the Amitsur complex for the associated graded Mackey functor  to $\cM$. We also note that
the process does not change the given $\bd_1\colon C_1 \to C_0$, but may change  $s_0$ in the first step.

First, let $\bd^\prime_{r+1} = \bd_{r+1}- s_{r-1}\circ \bd_{r}\circ\bd_{r+1}$. Then $\bd_{r}\circ \bd^\prime_{r+1}=0$ and if $\bd_{r}\circ\bd_{r+1}=0$ we have $\bd^\prime_{r+1} = \bd_{r+1}$. Note that both $\bd^\prime_{r+1}$ and $\bd_{r+1}$ preserve the induced filtration from $\cM$, and induce the same map on the  Amitsur complexes for the associated graded Mackey functor to $\cM$.

Next, we modify $s_{r}$. Let $\psi_{r} = s_{r-1}\circ \bd_r + \bd^\prime_{r+1}\circ s_r$. By construction, $\psi_{r}$ preserves the filtration and induces the identity on the associated graded. Hence, $\psi_r = 1_{C_r} + u$, where $u$ is nilpotent, and $\psi_r$ is invertible. Since
$\bd_r\circ\psi_r =\bd_r$, we can set $s^\prime_r= s_r\circ \psi_r^{-1}$ and obtain
$s_{r-1}\circ \bd_r + \bd^\prime_{r+1}\circ s^\prime_r = 1_{C_r}$ by pre-composing with $\psi_r$. Notice that if $s_r$ was already part of a chain contraction, then we do not alter it. It follows that $\bd^\prime_{r+1}= \bd^\prime_{r+1}\circ s^\prime_r\circ\bd^\prime_{r+1}$ and the induction step is complete. The naturality of $\bd_r^\prime$ and $\psi_r$ follow inductively from the explicit formulas. The naturality of $\psi_r$ implies the naturality of $s_r^\prime$ for use at the next step of the induction. Since no choices were involved in the construction of $(\bd^\prime_\ast, s^\prime_\ast)$, the new maps are canonically determined by the original data $(\bd_\ast, s_\ast)$.
\end{proof}
\begin{remark}   After this process, the new contractible complex gives an expression for $\cM(Y)$ as a direct summand of $\cM(X \times Y)$, with respect to the original induction map $\bd_1\colon\cM(X \times Y) \to  \cM(Y)$, and the new ``restriction" map $s_0' \colon \cM(Y) \to \colon\cM(X \times Y)$, since $\bd_1\circ s_0' = id$. 
In this situation, we say that $\cM(Y)$ is \emph{computed} from the family $\cF(X)$. If $\cM$ was actually a Mackey functor, computability is this sense would agree with the notion previously defined. Similar remarks apply to the contravariant version
$\cM^\ast(Am(X,Y), \delta^\ast)$.
\end{remark}

 We will also need a slight extension of this result. A \emph{filtered pre-complex} $(C, \bd)$ is a pre-complex of abelian groups equipped with a filtration 
$$C = F_0C \supset F_1C \supset F_2C \supset \dots$$
where each $F_iC$ is a pre-subcomplex of $(C, \bd)$, meaning that $\bd_r (F_iC_r) \subseteq F_iC_{r-1}$, for all $i, r$. 
We say that $(C,\bd)$ is a \emph{pseudo-complex} if the additional relation $\bd_r\circ \bd_{r+1} = 0$ holds, for all $r$, on each subquotient $F_iC/F_{i+1}C$.
We say that a pseudo-complex has a  \emph{pseudo-contraction} $s_\ast=(s_r)$ provided that $s_r(F_iC_r)\subseteq F_iC_{r+1}$, and $s_\ast$ induces an actual contraction on each subquotient complex $F_iC/F_{i+1}C$.

A pseudo-complex $(C, \bd)$ has a natural completion 
$$(C, \bd) \ \longrightarrow \ \invlim C/F_iC: = (\widehat C, \widehat \bd)$$
given by the inverse limit pre-complex with respect to the natural projections $C \to C/F_iC$, $i \geq 0$. A pseudo-contraction $s_\ast$ of $(C, \bd)$ induces a pre-contraction $\hat s_\ast$ of  $(\widehat C, \widehat \bd)$.

\begin{proposition}\label{pseudo_three} Let $(C, \bd)$ be a  pseudo-complex with filtration $\{F_iC\vv i \geq 0\}$.
 If $(C, \bd)$ admits a filtered pseudo-contraction $s_\ast$, then there exists canonical 
 data for which $ (\widehat C, \bd^\prime_\ast, s^\prime_\ast)$ is a contracted chain complex.   If the pseudo-complex was already a complex, $\bd^\prime_\ast = \hat \bd_\ast$, and if in addition $s_\ast$ was already a contraction, then $s^\prime_\ast=\hat s_\ast$. 
\end{proposition}
\begin{proof} The proof follows the same outline as for Proposition \ref{pseudo_two}, but we notice that the map $\psi_r = 1_{C_r} + u$ has the additional property that $u^{i+1} = 0$ on the quotient $C_r/F_{i}C_r$. This follows by induction from the exact sequences 
$$0 \to F_{i+1}C/F_iC \to F_0C/F_{i+1}C \to F_0C/F_iC \to 0$$
of pseudo-contractible complexes.
Then $\psi_r$ induces an invertible map on $ C_r/F_{i}C_r$, for each $i \geq 0$. We define $s^\prime_r = s_r\circ \psi^{-1}_r$ on $ C_r/F_{i}C_r$ as before. By induction, we have constructed contraction data
$(C/F_{i}C, \bd^\prime, s^\prime)$, for each $i \geq 0$. In addition, this contraction data is compatible with the projections $C/F_{i+1}C \to C/F_i$, and hence induce contraction data $(\widehat C,  \bd^\prime, s^\prime)$ for the inverse limit complex.
\end{proof}
\begin{remark} Once again, this process doesn't change $ \bd_1$, so
the new contractible complex gives an expression for $\widehat C_0$ as a direct summand of $\widehat C_1$, with respect to completion of  the original boundary map $\bd_1\colon C_1 \to  C_0$.
\end{remark}

\begin{example}[Whitehead groups]
Define the Whitehead groups, $Wh_n(\bZ G)$, as the homotopy groups of the
spectrum which is the cofibre of the Loday assembly map
$$BG^+\wedge K(\bZ) \to K(\bZ G).$$
The Loday assembly map is a map of bifunctors, 
\cite[Main Thm, p.~223]{nicas1},
and the Whitehead groups are bifunctors.
Furthermore, the $Wh_n$, $n\leq 3$, are Mackey functors, 
but it is not obvious from this description that the
other higher Whitehead groups are actually Mackey functors.
However, from the long exact sequence in homotopy theory we see that they 
are pseudo{-}Mackey functors.
From Example \ref{swan_ring}, Example \ref{cohomotopy} and Proposition \ref{pseudo_two}, we see that the $Wh_n(\bZ G)$ 
are computed by the hyperelementary family.
Similarly, the $Wh_n(\bZ G)\otimes \bZ_{(p)}$ are computed
by the $p${-}hyperelementary family (see  \cite{swan2},  \cite{lam1}, and \cite{nicas1} for partial results in this direction).

\end{example}
\begin{example}[Tate cohomology]
The Tate cohomology of $Wh_n$ or Quillen's $K_n$  are bifunctors
which are subquotients of $Wh_n$ or $K_n$, and hence are computed
by the hyperelementary family.
The localization maps $Wh_n\to Wh_n\otimes\bZ_{(2)}$ and
$K_n\to K_n\otimes\bZ_{(2)}$ induce isomorphisms on Tate cohomology.
Hence the Tate cohomology is  computed by the
2{-}hyperelementary family.
Given any pseudo{-}Mackey subfunctor of $Wh_n$ or $K_n$
which is invariant under the involution, we can form the 
Tate cohomology and this Tate cohomology functor is computed
by any family which contains the 2{-}hyperelementary family. \qed
\end{example}

\section{Surgery obstruction groups}\label{five}
In \cite[Theorem 1]{dress2}, Dress claims computability results for ``any of the $L$-functors defined by {C.~T.~C.~Wall" in \cite{wall-VI}}. However, the non-oriented $L$-groups $L_n(\bZ G, \omega)$ are not always Mackey functors, and so the techniques described by Dress in \cite{dress2}   do not appear to be adequate to prove the result in this generality. The point is that an inner automorphism by an element $g\in G$ with $\omega(g) = -1$ induces multiplication by $-1$ (which may not be the identity) on $L_n(\bZ G, \omega)$
(see \cite{taylor1}).
 One of the main applications of our more general techniques is to supply a proof that non-oriented $L$-theory is hyperelementary computable,   in the sense that $L_n(\bZ G,\omega)$ is the limit of restrictions or inductions involving hyperelementary subgroups of $G$.

Fix a finite group $G$, and the geometric antistructure for 
which $\theta = id$ 
and $b = e \in G$, \cite[1.B.3]{htw3}. 
Let $\omega\colon G \to \{\pm 1\}$ be a fixed orientation homomorphism, and for each subgroup $H\subset G$ let 
$\omega_H=  \omega|_H$. 
We define the following  categories:
\begin{enumerate}
\item $\bA(G,\omega)$, with objects finite groups $H$ isomorphic to some subgroup of  $G$, and morphisms
given by a Grothendieck group construction on finite biset forms $(X, \omega_X)$ (see \cite[p.~256]{htw3} for the definition). We construct $\bAdot(G,\omega)$ by taking the additive completion.
\item $\RbarMorita$, with objects and morphisms as defined in \cite[1.B.2]{htw3}, and the quotient category $\RbarWitt$ from \cite[1.C.2]{htw3}, for any commutative ring $R$ with unit.
\item $\RGomegaMorita$, with objects  $H$ isomorphic to some subquotient $K/N$ of $G$, with $N \subset \ker \omega$, and morphisms given by the Grothendieck group construction   on finite biset forms $(X, \omega_X)$, modulo an equivalence relation, as defined in \cite[1.B.3]{htw3}. We can define
the analogous quotient category $\RGomegaWitt$ by setting metabolic forms (see \cite[p.~254]{htw3}) to zero in the morphisms.
\end{enumerate}
Notice that by forgetting the orientation map $\omega$ we get functors into the categories discussed in Section \ref{three}. The construction of Definition \ref{DtoBurnside} gives a pre-bifunctor
$$j \colon \cD(G) \to \bAdot(G, \omega)$$
extending the pre-functor $\Or(G) \to \bAdot(G, \omega)$ out of the orbit category, defined on objects by $G/H \mapsto H$ and on morphisms by sending the $G$-map $f\colon G/H \to G/K$, given by $f(eH) =gK$,
 to the biset form $(\biset{K}{K}{g^{-1}Kg}, \omega_K)$. This definition depends on the choice of  coset representative $g$ for the  morphism $f$ in $\Or(G)$, since this time, if $x\in K$ and $\omega(x) = -1$, the two morphisms $eH \mapsto gK$ and $eH \mapsto gxK$ are sent to different biset forms.

\begin{lemma}\label{pretomackey}
In $\bAdot(G,\omega)$, the morphism $\left [\biset{H}{H}{x^{-1}Hx}, \omega_H\right ] = \omega(x)\cdot id$ for all $x\in H$.
If $F\colon \bAdot(G, \omega)\to \Ab$ is an additive functor, then $$F\circ j \colon \bAdot(G, \omega) \to \Ab$$ is a Mackey pre-functor, which is a Mackey functor if and only if all the inner automorphism morphisms
$F({_H}H_{x^{-1}Hx}, \omega_H) = id$, for all $x\in H$.
\end{lemma}
\begin{proof}  The identity morphism in $\bAdot(G,\omega)$ is represented by the biset form $(\biset{H}{H}{H}, \omega_H)$.
The map $\psi\colon _HH_H \to {_H}H_{xHx^{-1}}$ of biset forms defined by $\psi(h) =hx^{-1}$, see (\ref{bisetbijection}), induces an isometry of biset forms
$(\biset{H}{H}{x^{-1}Hx}, \omega_H)\cong  (\biset{H}{H}{H}, \omega(x)\cdot\omega_H)$
and hence 
$$\left [\biset{H}{H}{x^{-1}Hx}, \omega_H\right ] = \omega(x)\cdot id$$
in the Grothendieck group of morphisms of $\bAdot(G,\omega)$.

The property (M1) depends on conjugations acting trivially, or in other words,should induce $F(\psi) = id$ for all $x\in H$ (including those with $\omega(x) = -1$). 
\end{proof}

The $R$-group ring functor of \cite[1.B.4]{htw3} induces
a functor from
$\bAdot(G, \omega)$ to $\RGomegaMorita$ or further into $\RbarWitt$.
The required formulas are in section 1.B of \cite{htw3},
including the remark that since our morphisms are formed via
a Grothendieck construction, we are entitled to equate metabolics
on isomorphic modules.
There is a functor $a\colon \bAdot(G,\omega) \to \RGomegaMorita$ as before, and we let   $$d\colon \cD(G) \to \RGomegaMorita$$ be the pre-bifunctor $d=a\circ j$.
There is a homomorphism from the Dress ring 
$$GU(H, \bZ)  \to \Hom_{\RbarMorita}(H, H)$$
given by tensor product (see \cite{dress2} where it is
asserted that $GU(G, \bZ )$ acts on $L${-}theory, or   
\cite[p.~143]{hrt1} for explicit formulas).
Dress \cite{dress2} showed that the 
the hyperelementary 
family contracts the Dress ring. 
We observe that the same formulas give an action of
the Burnside quotient Green ring $\cA_{GU}$ on 
$\RGomegaMorita$.

 \begin{theorem}\label{thmc}  
  Let $F\colon \RGomegaMorita \to \Ab$ be an additive functor. 
 Then $F\circ d\colon \cD(G) \to \Ab$ is a Mackey pre-functor, and
  the $2$-adic completion of any such
 Mackey pre-functor is $2$-hyperelementary computable.
  If $\cM = F\circ d$ is a Mackey functor, then $\cM$ is  hyperelementary computable.
\end{theorem}
\begin{proof}   In the oriented case ($\omega\equiv 1$) the pre-functor $\cM$ is actually a Mackey functor, by Proposition \ref{pretomackey}. More generally, whenever $\cM = F\circ d$ is a Mackey functor the result follows as in Theorem \ref{rgmorita}, since $\cM$ is a Green module over $\cA_{GU}$. By \cite[Theorem 3]{dress2}, and Theorem \ref{main_theorem}, the Burnside quotient Green ring of the Dress ring is hyperelementary computable. 

In the non-oriented case, we define a filtration $F_i = 2^iF$, $i \geq 0$, with $F_0 = F$, and note that the sub-quotients $(F_i/F_{i+1})\circ d$ are Mackey functors. Now we let $(C,\bd)$ denote the filtered Amitsur pseudo-complex for $F_\ast\circ d$ with respect to $2$-hyperelementary induction, and the result follows from  Proposition \ref{pseudo_three}.  Notice that the passage from a pseudo-contractible pseudo-complex to a contractible complex does not change the first boundary map, so $F\circ d$ is 2-adically detected (generated) by the given restriction (induction) maps  to the $2$-hyperelementary subgroups. 
\end{proof}
\begin{example}[Non-oriented $L$-theory]\label{nonorientedL}
 The main example for us is the surgery obstruction group $L_n(\bZ G, \omega)$. It is a foundational result of Wall \cite{wall-V} that the surgery obstruction groups for finite groups are finitely-generated, with $2$-primary torsion exponent. Theorem \ref{thmc} computes $L_n(\bZ G, \omega)\otimes \Zhat_2$ as a limit (and as a colimit) over the $2$-hyperelementary subgroups $H \subset G$, $H \in \cH$. These limits use the standard induction or restriction maps, e.g.~for induction we have the surjective map
 $$\bd_1\colon \bigoplus_{H\in \cH} L_n(\bZ H, \omega)\otimes \Zhat_2 \to L_n(\bZ G, \omega)\otimes \Zhat_2$$
 and our contraction data gives the relation subgroup $\ker\bd_1 = \Image \bd^\prime_2$.
 
  We conclude that $L_n(\bZ G, \omega)$ is also effectively $2$-hyperelementary computable: the torsion subgroup is isomorphic to that of $L_n(\bZ G, \omega)\otimes \Zhat_2$, and the divisibility of the signatures is computable since the kernel and cokernel of the natural transformation
$$L_n(\bZ G, \omega) \to L_n(\bbR G, \omega)$$
of pseudo-Mackey functors are both $2$-primary torsion groups (see \cite[7.3, 7.4]{wall-V}).
The groups $L_n(\bbR G, \omega)$ were computed explicitly by Wall
\cite[2.2.1]{wall-VI} in terms of the irreducible characters of $G$.
The proof of computability given here applies in the oriented case ($\omega\equiv 1$), but in that case the $L$-group is a Mackey functor and the argument is essentially the same as the one given by Dress. Other important examples were listed in \cite[1.B.8]{htw3}.
\end{example}

\begin{example}[$L$-theory with decorations]
Let $R$ be a commutative ring with unit, and consider any $L${-}group
$L^B_n(RG,\omega)$ for $RG$ with anti{-}involution given
 by $\omega\colon G\to \{\pm 1\}$
with decoration in any involution invariant sub{-}bifunctor, $B$, 
of $K_i(\bZ G)$ or $Wh_i(\bZ G)$, $i \leq 1$ 
(see \cite{hrt1}  for a summary of 
the definitions).
It was checked in
\cite[Thm. 5.3, Cor. 5.5 and Ex. 5.14]{hrt1} that the corresponding
round $L${-}theories are functors out of $(\bZ G,\omega)${-}Morita.
Hence these $L${-}theories are pseudo-Mackey functors and are contracted
by the hyperelementary family.
It was also checked in 
\cite[Prop 5.6, Cor. 5.7 and Ex. 5.14]{hrt1} that the
corresponding ordinary $L${-}theories are functors out
of $(\bZ G,\omega)${-}Morita, so the same computation result holds.
\end{example}

\begin{example}[Localization]
Dress \cite{dress2}  shows that the Dress ring $GU$  is contracted by
any family containing the 2{-}hyperelementary and $p${-}elementary families.
More precisely, he showed that the 2{-}localization of the Dress
ring is contracted by the 2{-}hyperelementary family,
and the $p${-}localization, $p$ odd, is contracted by
the $p${-}elementary family.

Proposition \ref{pseudo_two} and a standard ``mixing" argument shows that
this smaller family  suffices
to contract the  $L^B$ functors described above.
For sub{-}bifunctors $B$ closed under the action
of the Dress ring, this was proved by Dress \cite{dress2} and Wall
\cite{wall-VI}.
A similar argument shows that the odd{-}dimensional $L^B${-}groups
are contracted by the 2{-}hyperelementary  family alone.
\end{example}

\begin{example}[Symmetric, hyper-quadratic and lower $L$-theory]
The Ranicki symmetric and hyper{-}quadratic $L${-}theories
\cite{ra19} are
also functors out of $(\bZ G,\omega)${-}Morita  and hence are contracted by
the hyperelementary family.
The hyper{-}quadratic theory is a 2{-}torsion group with an exponent
so it is contracted by the 2{-}hyperelementary family
(as above, we note that the 2{-}localization map induces an
isomorphism on this functor and use the 2{-}local contraction
of functors out of $(\bZ, -)${-}Morita by the 2{-}hyperelementary
family).
The lower $L${-}theories for a ring with anti{-}structure
can be defined in terms of the $L${-}theory of the ring with
some Laurent variables adjoined \cite{ra19} and hence are functors
out of $(\bZ G,\omega)${-}Morita. 
Therefore $L${-}theories with decorations 
in sub{-}Mackey functors of $K_i$ for $i < 0$ 
are contracted by the hyperelementary family.
The higher $L${-}theories of Weiss and Williams \cite{weiss-williams2} should also be
amenable to these techniques.
\end{example}
\begin{example}[Farrell-Hsiang induction]
The technique of \cite[\S\S 1-2]{farrell-hsiang2} was originally introduced to apply induction theory to the $L$-groups of an infinite group $\Gamma$.
Let $\pr \colon \Gamma \to G$ be a homomorphism to a finite group $G$, and define an orientation character for $\Gamma$ by the composition $\omega\circ \pr$, where 
$\omega\colon G \to \{\pm 1\}$ is an orientation character for $G$. 
Then
$L^B_{\ast}(R\Gamma,\omega)$ is an additive functor $\RGomegaWitt \to \Ab$, which defines a pseudo-Mackey functor as above. To check this, note that
 we again have a generating set for the morphisms consisting of the bisets $X=H_2\times_K  H_1$, where $K \subset H_2\times H_1$ is a subgroup. To produce the needed biform on $X$, we adapt the formulas in \cite[1.B]{htw3} with $\theta_X = id$.
 If $\omega\equiv 1$, it follows that these $L$-groups can be computed 
 in terms of the
$L${-}theory of the various subgroups $\Gamma_H = \pr^{-1}(H)$, $H\subset G$. In particular, it is enough to use the  hyperelementary subgroups $H$ of $G$. 
\end{example}

%\bibliographystyle{ih}
%\bibliography{ihmain}
%\end{document}
%%%%%%%%%%%%%%%%%%%%%%%%%
%%%%%%%%%%%%%%%%%%%%%%%%%
\providecommand{\bysame}{\leavevmode\hbox to3em{\hrulefill}\thinspace}
\providecommand{\MR}{\relax\ifhmode\unskip\space\fi MR }
\providecommand{\MRhref}[2]{%
  \href{http://www.ams.org/mathscinet-getitem?mr=#1}{#2}
}
\providecommand{\href}[2]{#2}

\end{document}